\documentclass[a4paper,11pt]{article}

\usepackage[english]{babel} 
\usepackage[utf8x]{inputenc} 
\usepackage{amsmath} 
\usepackage{amsthm} 
\usepackage{amssymb} 
\usepackage{mathtools} 
\usepackage{dsfont} 
\usepackage[margin=1.2in]{geometry} 

\title{On the extraordinary construction of cycle sets by Wolfgang Rump}
\author{Pravin Bhandari, Miguel Córdoba Esteve,\\ Jamie Henderson, Scott Warrander}

\DeclareMathOperator{\N}{\mathbb{N}}
\DeclareMathOperator{\Z}{\mathbb{Z}}
\DeclareMathOperator{\R}{\mathbb{R}}
\DeclareMathOperator{\F}{\mathbb{F}}
\DeclareMathOperator{\id}{\text{id}}

\newtheorem{theorem}{Theorem}[subsection]
\newtheorem{lemma}[theorem]{Lemma}
\newtheorem{proposition}[theorem]{Proposition}
\newtheorem{corollary}[theorem]{Corollary}
\theoremstyle{definition} \newtheorem{definition}[theorem]{Definition}
\theoremstyle{definition} \newtheorem{example}[theorem]{Example}
\theoremstyle{definition} 

\numberwithin{equation}{section}

\begin{document}

\maketitle

\begin{abstract}
Cycle sets are algebraic structures introduced by Rump to study set theoretic solutions to the Yang-Baxter equation. While studying cycle sets Rump also introduced braces, which have since overtaken cycle sets as a tool for studying solutions. This survey paper is primarily an introduction to cycle sets, motivating their study and relating them to key results of brace theory and Yang-Baxter theory. It is aimed at anyone from those already very familiar with braces but less familiar with cycle sets, to those with only a basic level of background in ring theory and group theory. We introduce cycle sets following Rump's original results - giving more detailed, easy to follow versions of his proofs - and then relate them back to left braces. We also go on to discuss interesting constructions of cycle sets which do not necessarily correspond directly to braces.
\end{abstract}

\section*{Introduction}
\addcontentsline{toc}{chapter}{Introduction}

The Yang-Baxter equation is central to many areas of mathematical physics, such as statistical mechanics and quantum groups. Solutions take the form of a map $R:V \otimes V \to V \otimes V$, where $V$ is a vector space, however in 1992, Drinfeld \cite{d} proposed the study of solutions obtained by linear extension from a so called set theoretic solution $r:X^2 \to X^2$ on a basis $X$. It is these set theoretic solutions with which we are concerned.

Braces are an algebraic structure generalising Jacobson radical rings, introduced by Rump \cite{ru2} to study set theoretic solutions to the Yang-Baxter equation. Along with their associated solutions, braces are now studied in their own right, with modern brace theorists focussing on left braces. However, Rump originally introduced right braces as an algebraic structure associated to a linear cycle set, a special type of cycle set.

\subsection*{Overview of the Paper}

In chapter 1 we give background on the ring theory needed to motivate braces, and in chapter 2 we introduce the Yang-Baxter equation and give initial definitions and results on (left) braces. Readers already familiar with braces and the Yang Baxter equation may want to skip these sections.

In chapter 3 we present results from \cite{ru1}, where Rump introduces cycle sets, and \cite{ru2}, where he introduces braces. Rump's results are often either stated without proof or proved with minimal detail, so here we give much more detailed proofs to make the results more accessible.

In chapter 4 we use the results from chapter 3 to prove that braces give rise to solutions to the Yang-Baxter equation, and go on to describe how braces and cycle sets allow us to study the structure group of an arbitrary solution.

In chapter 5 we motivate the independent study of cycle sets by using them to find finite solutions of prime power order, which do not arise from braces \cite{cpr}.

\subsection*{Our Contributions}

Although we do not present any significant original results, we have taken advantage of the connections between the theories of braces and cycle set to give some improved formulae and more efficient proofs of known results.

We have also given rigorous treatments to results which had previously only been published with sketches of proofs, or stated as "obvious". While this is understandable in a high level research paper, we have aimed to make the proofs accessible to those with less experience (although perhaps with more time on their hands). This particularly applies to chapter 3.

While brace theory continues to slowly grow in popularity, relatively few people are working on cycle sets. We hope that this may serve as an introduction to the topic, which is accessible to those at the undergraduate level, but also with enough depth to interest those already involved in researching braces and the Yang-Baxter equation.

\section{Ring Theory}

\subsection{Nil and nilpotent rings}

We begin by recalling the definition of a ring:
\begin{definition}
Let $R$ be a set and $+, *$ be binary operations on $R$. We say that $(R,+,*)$ is a \emph{ring} iff:
\begin{enumerate}
\item $(R,+)$ is an abelian group,
\item $*$ is associative, so for $a,b,c \in R$ we have $a*(b*c) = (a*b)*c$,
\item $*$ distributes over $+$, so $a*(b+c) = a*b + a*c$ and $(a+b)*c = a*c + b*c$.
\end{enumerate}

If there exists some $1_R \in R$ such that $1_R a = a 1_R = a$ for all $a \in R$ we say that $R$ is a \emph{ring with identity}, and call $1_R$ the \emph{identity element}.

If for all $a,b \in R$ we have $a*b = b*a$ we say that $R$ is \emph{commutative}.
\end{definition}

When it is clear that we are referring to the ring structure on $R$, we will usually write $R$ rather than $(R,+,*)$, using the latter when there are multiple structures we can define on $R$. Also, whenever it will not give rise to ambiguities, we abbreviate $a*b$ to $ab$.

\begin{definition}
Let $R$ be a ring and $a \in R$. We say that $a$ is \emph{nilpotent} iff there is some $n \in \N$ such that $a^n = \prod_{i=1}^n a = 0$.
\end{definition}

\begin{definition}
We say that a ring $R$ is a \emph{nil ring} iff every $r \in R$ is nilpotent.
\end{definition}

\begin{definition}
We say that a ring $R$ is a \emph{nilpotent ring} if there exists some $n \in \N$ such that for every $(a_1, \cdots, a_n) \in R^n$ we have $a_1 \cdots a_n = 0$.
\end{definition}

\begin{example}
Let $R$ be a ring of strictly upper triangular $3 \times 3$ matrices over a field $\F$:

\[ R = \left\{ \begin{bmatrix}
0 & a_1 & a_2\\
0 & 0  & a_3\\
0 & 0 & 0
\end{bmatrix} : a_i \in \F \right\}. \]
Then $R$ is nilpotent and also nil, since every product of 3 elements is 0:
\begin{align*}
&\begin{pmatrix}
0 & a_1 & a_2\\
0 & 0  & a_3\\
0 & 0 & 0
\end{pmatrix}\cdot
\begin{pmatrix}
0 & b_1 & b_2\\
0 & 0  & b_3\\
0 & 0 & 0
\end{pmatrix}\cdot
\begin{pmatrix}
0 & c_1 & c_2\\
0 & 0  & c_3\\
0 & 0 & 0
\end{pmatrix}\\
&=
\begin{pmatrix}
0 & a_1 & a_2\\
0 & 0  & a_3\\
0 & 0 & 0
\end{pmatrix}\cdot
\begin{pmatrix}
0 & 0 & b_1 c_3\\
0 & 0  & 0\\
0 & 0 & 0
\end{pmatrix}
=
\begin{pmatrix}
0 & 0 & 0\\
0 & 0 & 0\\
0 & 0 & 0
\end{pmatrix}.
\end{align*}
\end{example}

In the previous example we had a ring that was both nil and nilpotent. In fact we can clearly see that any nilpotent ring is also nil: if $R,n$ satisfies definition 1.4, then $a^n = 0$ for all $a \in R$, so every element of $R$ is nilpotent. It is not the case however that every nil ring is nilpotent

\begin{example}
Let $T$ be the set of infinite matrices with entries from $\N$. Let $S \subset T$ be the set such that $S$ has a finite number non-zero entries. Then let $R \subset S$ be the subset containing all strictly upper triangular matrices. Then $R$ is a ring, and all $a \in R$ we have that $a^n = 0$ for some $n \in \N$. Therefor $R$ is nil, but it is not nilpotent. 
\end{example}

\subsection{Jacobson radial rings}

\begin{definition}
A ring $R$ is a \emph{Jacobson radical ring} iff for every $a \in R$ there exists some $b \in R$ such that $a + b + ab = 0$.
\end{definition}

We notice that if $R$ is a unital ring, then $-1 \in R$ and $-1 + b + (-1)b = -1 + b - b = -1$, so no Jacobson radical ring is unital. However, when working with a Jacobson radical ring it is often convenient to embed it into a unital ring.

\begin{proposition}
Let $R$ be a ring. Then there exists a unital ring $R^1$ with a natural embedding $R \hookrightarrow R^1$, so that $R$ is Jacobson radical if and only if for all $a \in R$ the element $1 + a \in R^1$ has an inverse of the form $1 + b$ for some $b \in R$.
\end{proposition}

\begin{proof}
Let $R^1 = \Z \times R$, with addition defined element wise and multiplication defined by:
\[ (n,a)*(m,b) = (nm, ma + nb + ab) \]
for $n,m \in \Z$, $a,b \in R$ (where $na = \sum_{i=1}^n a$). The map $a \mapsto (0,a)$ embeds $R$ in $R^1$ since
\[(0,a)*(0,b) = (0, 0a + 0b + ab) = (0, ab),\]
and $(1,0)$ is the identity in $R^1$ since 
\[(1,0)*(n,a) = (1n, 1a + n0 + a0) = (n,a),\]
\[(n,a)*(1,0) = (n1, n0 + 1a + 0a) = (n,a).\]

Now, let $a \in R$. If there exists $b \in R$ such that $a + b +ab =0$, then
\[ (1+a)*(1+b) = (1,a)*(1,b) = (1, a + b + ab) = (1,0), \]
so $(1+a)^{-1} = (1+b)$.

Conversely, suppose $(1+a)*(1+b) = 1$. Then:
\[ (1+a)*(1+b) = 1 + a + b + ab = 1 \Rightarrow a+b+ab = 0. \]
\end{proof}

We see that Jacobson radical rings generalise nil and nilpotent rings.

\begin{proposition}
If $R$ is a nil or nilpotent ring, then $R$ is Jacobson radical.
\end{proposition}

\begin{proof}
Since every nilpotent ring is nil, it suffices to show that if $R$ is nil then $R$ is Jacobson radical.

We claim that for $a \in R$, $b = \sum_{n=1}^\infty (-1)^n a^n$ satisfies $a + b + ab = 0$. $b$ exists since $a$ is nilpotent so for some $N$ we have $a^n = 0$ whenever $n \geq N$, and indeed:
\begin{align*}
a + b + ab &= a + \sum_{n=1}^\infty (-1)^n a^n + a\sum_{n=1}^\infty (-1)^n a^n\\
&= \sum_{n=2}^\infty (-1)^n a^n + \sum_{n=1}^\infty (-1)^n a^{n+1}\\
&= \sum_{n=2}^\infty (-1)^n a^n + \sum_{n=2}^\infty (-1)^{n-1} a^n\\
&= \sum_{n=2}^\infty ((-1)^n + (-1)^{n-1})a^n\\
&= 0.
\end{align*}
\end{proof}

However, as suggested when we said that Jacobson radical rings \emph{generalise} nil and nilpotent rings, there exist Jacobson radical rings which are not nil.

\begin{example}
Let $\R[[x]]$ be the ring of formal power series over the real numbers, and let $S$ be a subring of $\R[[x]]$ consisting of power series with zero constant term. Every power series $f(x) \in \R[[x]]$ with non-zero constant term is invertible in $\R[[x]]$, so $1 + f(x)$ is invertible for every $f(x) \in S$. Thus by proposition 1.2.2 $S$ is Jacobson radical. However S is clearly not nil, as $f(x) = x \in S$ but $(f(x))^n = x^n \neq 0$ for all $n \in \N$. Since every nilpotent ring is nil, $S$ is also not nilpotent.
\end{example}

Motivated by the definition of a Jacobson radical ring, we define an additional binary operation on a ring:

\begin{definition}
Let $R$ be a ring. We define the \emph{adjoint multiplication} $\circ$ on $R$ by:
\[ a \circ b = a + b + ab \quad \text{for } a,b \in R.\]
\end{definition} 

So $R$ is a Jacobson radical ring if and only if for every $a \in R$ there is a $b \in R$ such that $a \circ b = 0$. Since $a \circ 0 = a + 0 + a0 = a$, and $0 \circ a = 0 + a + 0a = a$, we can think of 0 as the \emph{adjoint identity}, and so a ring is Jacobson radical if and only if element has an \emph{adjoint inverse}. A nice property for inverse elements to have is uniqueness, which turns out to hold for the adjoint in a Jacobson radical ring:

\begin{lemma}
Let $R$ be a Jacobson radical ring and $a \in R$. Then there exists a \textbf{unique} $b \in R$ such that $a \circ b = 0$ and $b \circ a = 0$.
\end{lemma}

\begin{proof}
We know that since $R$ is Jacobson radical there exists $b$ such that $a \circ b = 0$. Suppose that $b \circ c = 0$, so embedding $R$ into a unital ring $R^1$ as in proposition 1.2.2, we have:
\[ 1+c = [(1+a)(1+b)](1+c) = (1+a)[(1+b)(1+c)] = 1+a, \]
so $a = c$ meaning $b \circ a = a \circ b = 0$.

Suppose that $a \circ b = 0$ and $a \circ b' = 0$. Embedding $R$ into a unital ring $R^1$, we have:
\[ 1 + b = (1+b)[(1+a)(1+b')] = [(1+b)(1+a)](1+b') = 1+b',\]
so $b = b'$ giving uniqueness.
\end{proof}

\begin{theorem}
A ring $R$ is Jacobson radical if and only if $(R,\circ)$ is a group. We call $R^\circ = (R,\circ)$ the \emph{adjoint group} of $R$.
\end{theorem}

\begin{proof}
In our discussion above we saw that 0 is the identity in $(R,\circ)$, and by lemma 1.2.6, if $R$ is Jacobson radical then every $a \in R$ has a unique inverse in $(R, \circ)$. It remains to show that $\circ$ is associative:
\begin{align*}
a \circ (b \circ c) &= a + (b + c + bc) + a(b+c+bc)\\
&= a + b + c + bc + ab + ac + abc\\
&= (a + b + ab) + c + (a + b + ab)c\\
&= (a\circ b)\circ c.
\end{align*} 
\end{proof}

\section{Braces and the Yang-Baxter Equation}

\subsection{The Yang-Baxter equation}

The quantum Yang-Baxter equation (QYBE) first appeared in independent studies of 2D integrable systems by McGuire, in 1964, and Yang, in 1967. Since then the equation has garnered immense interest, and researchers have revealed its connection with many areas of mathematics and physics (knot theory, $C^*$ algebras, statistical mechanics, 2D conformal field theory, quantum computing etc.).

A 2D quantum integrable system is a 2D multi-particle scattering system where particles interact with each other. One useful property of these systems is that they can always be decomposed into models consisting of only 3 particles. The QYBE imposes the following condition on these 3-particle scattering models: the order in which the 2-particle scatterings are performed inside the 3-particle scattering model is inconsequential. Thus, for certain 3-particle systems in which the initial and final state are equal, it does not matter in what order the particles interact pair-wisely with each other to get to the final state.

One way to show this equivalence imposed by the QYBE is that if you have an initial state $(a,b,c)$ and a final state $(c,b,a)$ on 3 particles, then, by only performing permutations of adjacent letters, each possible way of going from one to the other is equivalent:
\begin{align}
    (a,b,c) \rightarrow (a,c,b) \rightarrow (c,a,b) \rightarrow (c,b,a)\\
    (a,b,c) \rightarrow (b,a,c) \rightarrow (b,c,a) \rightarrow (c,b,a)
\end{align}
hence 2.1 and 2.2 are equivalent. These permutations can be expressed as a composition of some matrices $r_{ij}$, which are generated from $r$, the R-matrix corresponding to the integrable system. This R-matrix has dimension $n^2$, where $n$ is the number of degrees of freedom of each particle inside the model. (for further info on this derivation see \cite{y}, \cite{sb})
\begin{definition}
Let $V$ be a vector space and $R:V \otimes V \to V \otimes V$ be a linear map. We say that $(V,R)$ is a solution to the \emph{Yang-Baxter equation} iff:
\begin{equation}
(R \otimes \text{id})(\text{id} \otimes R)(R \otimes \text{id}) = (\text{id} \otimes R)(R \otimes \text{id})(\text{id} \otimes R).
\end{equation}
\end{definition}

For many years, most known solutions were deformations of the identity solution, and their related algebraic structures (Hopf Algebras) have been studied in depth. In 1992 however, Drinfeld proposed in \cite{d} that mathematicians should focus their efforts on studying a new class of solutions, obtained as follows:

\begin{definition}
Let $X$ be a set and $r:X^2 \to X^2$. We say that $(X,r)$ is a \emph{set theoretic solution} to the Yang-Baxter equation iff:
\begin{equation}
r_1 r_2 r_1 = r_2 r_1 r_2,
\end{equation}
where $r_1(x,y,z) = (r(x,y),z)$ and $r_2(x,y,z) = (x, r(y,z))$.
\end{definition}

A set theoretic solution on a set $X$ extends linearly to a solution on a $k$-vector space $V$ where $V = k^{(X)}$ for some field $k$. 

\begin{definition}
Let $(X,r)$ be a set theoretic solution to the Yang-Baxter equation, and write $r(x,y) = (\lambda_x(y), \tau_y(x)) = ({}^x y, x^y)$.
\begin{itemize}
\item We say that $(X,r)$ is \emph{non-degenerate} iff $\lambda_x, \tau_x$ are bijections for all $x \in X$,
\item We say that $(X,r)$ is \emph{involutive} iff $r^2 = \text{id}$.
\end{itemize}
\end{definition}

\noindent
\emph{Remark:} The notation $r(x,y) = (\lambda_x(y),\tau_y(x))$ is a less popular version of $r(x,y) = (\sigma_x(y),\tau_y(x))$, but we will avoid this notation since $\sigma_x$ will mean something else in chapter 3. The $r(x,y) = ({}^x y, x^y)$ notation is common when working with cycle sets and this is the notation we use in chapter 3, elsewhere we use $r(x,y) = (\lambda_x(y),\tau_y(x))$.

\begin{lemma}
A solution to the YBE $(X,r)$ with $r(x,y) = (\lambda_x(y), \tau_y(x))$ is involutive if and only if $\tau_y(x) = \lambda_{\lambda_x(y)}^{-1}(x)$.
\end{lemma}

\begin{proof}
We first see that
\begin{align*}
r(r(x,y)) = (x,y) &\Rightarrow r(\lambda_x(y), \tau_y(x)) = (x,y)\\
&\Rightarrow (\lambda_{\lambda_x(y)}(\tau_y(x)), \tau_{\tau_y(x)}(\lambda_x(y))) = (x,y)\\
&\Rightarrow \tau_y(x) = \lambda_{\lambda_x(y)}^{-1}(x),
\end{align*}
then that if $r(x,y) = (\lambda_x(y),\lambda_{\lambda_x(y)}^{-1}(x))$
\begin{align*}
r(r(x,y)) &= r(\lambda_x(y),\lambda_{\lambda_x(y)}^{-1}(x))\\
&= (\lambda_{\lambda_x(y)}(\lambda_{\lambda_x(y)}^{-1}(x)), \lambda_{\lambda_{\lambda_x(y)}(\lambda_{\lambda_x(y)}^{-1}(x))}^{-1}(\lambda_x(y)))\\
&= (x, \lambda_x^{-1}(\lambda_x(y)))\\
&= (x,y).
\end{align*}
\end{proof}

An important fact about Jacobson radical rings is that they naturally give rise to non-degenerate, involutive, set theoretic solutions to the YBE.

\begin{theorem}
Let $R$ be a Jacobson radical ring. Then $(R,r)$ is a non-degenerate involutive set theoretic solution to the YBE:
\begin{equation}
r(a,b) = (ab + b, ca + a) = (a \circ b - a, c \circ a - c),
\end{equation}
where $(a \circ b - a) \circ c = 0$.
\end{theorem}

\begin{proof}
This is a special case of theorem 2.2.8, which is proved in chapter 4.
\end{proof}

So every Jacobson radical ring gives rise to a non-degenerate involutive solution. It is in fact possible to further generalise Jacobson radical rings to obtain even more solutions, and this is will be the initial motivation of braces.

\subsection{Braces}

\begin{definition}
Let $B$ be a set with binary operations $+,\circ$ such that $(B,+)$ is an abelian group and $(B,\circ)$ is a group:
\begin{itemize}
\item $(B,+,\circ)$ is a \emph{left brace} iff \begin{equation}
a \circ(b+c) + a = a \circ b + a \circ c,
\end{equation}
\item $(B,+,\circ)$ is a \emph{right brace} iff \begin{equation}
(a+b)\circ c + c = a\circ c + b\circ c,
\end{equation}
\item $(B,+,\circ)$ is a \emph{two sided brace} iff $(B,+,\circ)$ is both a left brace and a right brace.
\end{itemize}
\end{definition}

\begin{example}
Any abelian group $(G,+)$ trivially gives a two sided brace $(G,+,+)$, since:
\[ a + (b + c) + a = (a + b) + (a + c) = (b+a) + (c+a) = (b + c) + a + a. \]

In fact, whenever we have a left or right brace $B$, and $(B,\circ)$ is abelian, then $B$ is a two sided brace.
\end{example}

We can go back and forth between left and right braces in the following sense:

\begin{lemma}
If $(B,+,\circ)$ is a left brace, then $(B,+,\circ^{op})$ is a right brace, where $\circ^{op}$ is the opposite multiplication $a \circ^{op} b = b \circ a$. We have that $op: (B,+,\circ) \mapsto (B,+,\circ^{op})$ is a bijection between left and right braces.
\end{lemma}

\begin{proof}
The first result follows from:
\[ a \circ (b + c) + a = a \circ b + a \circ c \iff (b+c) \circ^{op} a + a = b \circ^{op} a + c \circ^{op} a. \]

Since $(\circ^{op})^{op} = \circ$, it follows that $op$ is a bijection.
\end{proof}

We can define another operation on a brace, which will be useful for proving arithmetic facts about braces, and makes clear the connection to Jacobson radical rings.

\begin{definition}
Let $(B,+,\circ)$ be a left brace, and define the \emph{ring multiplication} $*$ on $B$ by:
\begin{equation}
 a*b = a \circ b - a - b
\end{equation}
for $a,b \in B$.
\end{definition}

\begin{proposition}
Every Jacobson radical ring $(R,+,*)$ is a two sided brace $(R,+,\circ)$ where $\circ$ is the adjoint multiplication.

Similarly, every two sided brace $(B,+,\circ)$ is a Jacobson radical ring $(B,+,*)$ where $*$ is the ring multiplication.
\end{proposition}

\begin{proof}
Since the adjoint multiplication on a ring and the ring multiplication on a brace are inversely defined, and a ring is Jacobson radical if and only if it is a group under the adjoint multiplication, it suffices to prove the following:
\[ a*(b+c) = a*b + a*c \iff a\circ(b+c) + a = a\circ b + a\circ c,\]
\[ (a+b)*c = a*c + b*c \iff (a+b)\circ c + c = a\circ c + b\circ c.\]

For the left distributivity:
\begin{align*}
&a\circ(b+c) + a = a\circ b + a\circ c\\
\iff& a + (b + c) + a*(b+c) + a =  a + b + a*b + a + c + a*c\\
\iff& a*(b+c) = a*b + a*c,
\end{align*}
and for the right:
\begin{align*}
&(a+b)\circ c + c = a\circ c + b\circ c\\
\iff& (a + b) + c + (a+b)*c + c =  a + c + a*c + b + c + b*c\\
\iff& (a+b)*c = a*c + b*c.
\end{align*}
\end{proof}

Now we give some basic facts about braces which will help our calculations:

\begin{lemma}
Let $B$ be a left [right] brace. Then for $a,b,c \in B$ the following hold:
\begin{align}
0 \circ a =& a \circ 0 = a,\\
a \circ (-b) = 2a - a\circ b &\quad [\, (-a) \circ b = 2b - a \circ b \,],\\
a \circ (b - c) - a = a \circ b - a \circ c &\quad [\, (a - b)\circ c - c = a\circ c + b\circ c\, ].
\end{align}
\end{lemma}

\begin{proof}
The statements on right braces are equivalent to those for left braces except using the opposite multiplication, so by lemma 2.2.2 it suffices to prove the results for left braces.

The first point says that $0 = 1$, where $1$ denotes the identity in $(B,\circ)$. We calculate $a \circ 0 + a = a \circ (0 + 0) + a = a \circ 0 + a \circ 0 + a$, so subtracting $a \circ 0$ gives $a = a\circ 0$. Left multiplying by $a^{-1}$ gives $1 = 0$.

For the second we use the ring multiplication:
\[ a\circ (-b) = a\circ(0 - b) = a*(-b) + a - b = 2a -(a*b + a + b) = 2a - a \circ b. \]

Finally, the third follows from the second:
\[ a \circ (b - c) - a = a \circ (b-c) + a - 2a = a\circ b + a \circ(-c) -2a = a \circ b - a\circ c. \]
\end{proof}

\hfill

Now, motivated by the set theoretic solution associated to a Jacobson radical ring, we define the following:

\begin{definition}
Let $B$ be a left brace. For each $a \in B$ define a map $\lambda_a:B \to B$ called the \emph{lambda map} by 
\begin{equation}
\lambda_a(b) = a \circ b - a.
\end{equation}
\end{definition}

\begin{lemma}
Let $B$ be a brace and define $\lambda:B \to Sym(B)$ by $a \mapsto \lambda_a$. Then $\lambda$ is a homomorphism of the adjoint group $B^{\circ}$, and each $\lambda_a$ is an automorphism of the additive group $(B,+)$.
\end{lemma}

\begin{proof}
We need to show that $\lambda_{a \circ b} = \lambda_a \lambda_b$. Using 2.11 we have
\begin{align*}
\lambda_a(\lambda_{b}(c)) &= \lambda_a(b \circ c - b)\\
&= a \circ (b \circ c - b) - a\\
&= a \circ b \circ c - a \circ b + a - a\\
&= (a \circ b) \circ c - (a \circ b)\\
&= \lambda_{a \circ b}(c).
\end{align*}

Next we need to show that $\lambda_a(b + c) = \lambda_a(b) + \lambda_a(c)$:
\begin{align*}
\lambda_a(b + c) &= a \circ (b + c) - a\\
&= a \circ b + a \circ c - 2a\\
&= (a \circ b - a) + (a \circ c - a)\\
&= \lambda_a(b) + \lambda_a(c).
\end{align*}
\end{proof}

Now we can state one of the most important theorems on Braces, which we prove over the next two sections.

\begin{theorem}
Let $B$ be a left brace, and let $r:B^2 \to B^2$ be defined by:
\begin{equation}
r(a,b) = \left(\lambda_a(b), \lambda_{\lambda_{a}(b)}^{-1}(a)\right).
\end{equation}
Then $(B,r)$ is a non-degenerate involutive set theoretic solution to the YBE.
\end{theorem}

As noted earlier, theorem 2.1.4 is a special case of this, since $c$ satisfying $(a \circ b - a) \circ c = \lambda_a(b) \circ c = 0$ is clearly $\lambda_a(b)^{-1}$ (the inverse in the adjoint group), and $\lambda_{\lambda_a(b)^{-1}} = \lambda_{\lambda_a(b)}^{-1}$ by the lemma above. This theorem also makes clear the relation between the lambda map notation and the notation $r(x,y) = (\lambda_x(y),\tau_y(x))$ for a general solution.

\section{The Results of Rump}

Wolfgang Rump invented braces when studying the Yang-Baxter equation. Although most brace theorists use left braces for various reasons such as links to braided groups, when Rump originally introduced the concept he used right braces. In their proofs he and his collaborators also use cycle sets, another structure introduced by Rump. In this section we introduce cycle sets, giving many of Rump's early results and tracing the origin of braces from cycle sets via linear cycle sets. Rump states many of his results without proofs, or with short sketched proofs, so we will give very detailed and comprehensive proofs of all the results we use.

\subsection{The Quantum Yang-Baxter Equation}

We said that Rump invented braces while studying the Yang-Baxter equation, however in definition 2.1.1 what we defined was the \emph{braid} version of the Yang-Baxter equation preferred by brace theorists. Rump was using the version of the Yang-Baxter equation preferred by physicists, which we will refer to as the \emph{quantum} Yang-Baxter equation.

Note that our original definition was not of the \emph{classical} Yang-Baxter equation: there is in fact a classical version of the equation which we are not concerned with in this paper. Our terminology mirrors that of the literature, in the sense that papers using the braid version tend to be in pure mathematics and just call it the Yang-Baxter equation, whereas papers concerned with applications to physics tend to specify "quantum".

\begin{definition}
Let $X$ be a set and $R:X^2 \to X^2$. We say that $(X,r)$ is a set theoretic solution to the quantum Yang-Baxter equation iff:
\[ R_{12} R_{13} R_{23} = R_{23} R_{13} R_{12}, \]
where $R_{ij}:X^3 \to X^3$ acts as $R$ on the $i$th and $j$th components (in that order), an as the identity on the third.
\end{definition}

\noindent
\emph{Remark:} We adopt the convention of using $R$ for solutions to the quantum equation, and $r$ for solutions to the braid equation. When abbreviating we will use QYBE to refer to the quantum equation, and YBE will be reserved for the braid equation.

\begin{definition}
Let $(X,R)$ be a set theoretic solution to the quantum Yang-Baxter equation, $R(x,y) = (x^y, {}^x y)$.
\begin{itemize}
\item We say that $(X,R)$ is \emph{non-degenerate} iff $x \mapsto x^y$ and $x \mapsto {}^y x$ are bijections for all $y \in X$,
\item We say that $(X,R)$ is \emph{unitary} iff $(R_{21})^2 = \text{id}$.
\end{itemize}
\end{definition}

We can go back and forth between solutions to the quantum and braid equations by the following proposition:

\begin{proposition}[{\cite[Prop. 1.2]{ess}}]
Let $(X,R)$ be a set theoretic solution to the QYBE, and $p:X^2 \to X^2$ be given by $p(x,y) = (y,x)$. Then $(X,pR)$ is a set theoretic solution to the YBE. Furthermore this is a bijection which preserves non-degeneracy, and restricts to a bijection between unitary solutions and involutive solutions.
\end{proposition}

We now give a more explicit characterisation of non-degenerate, unitary, set theoretic solutions, which will be useful in our proofs in the next section. (Rump states these formulae in \cite{ru1} but does not prove them.)

\begin{lemma}
Let $X$ be a set, and $R:X^2 \to X^2$ be denoted by $R(x,y) = (x^y, {}^x y)$. Then $(X,R)$ is a non-degenerate unitary solution to the QYBE if and only if:
\begin{enumerate}
\item (Non-degenerate) The maps $x \mapsto {}^x y$, $y \mapsto y^x$ are bijections,
\item (Unitary) The following holds:
\begin{align}
{}^{({}^x y)}(x^y) &= x,\\
{({}^x y)}^{(x^y)} &= y.
\end{align}
\item (Solution to the QYBE) The following holds:
\begin{align}
(x^y)^z &= (x^{({}^y z)})^{(y^z)},\\
{}^{(x^{({}^y z)})}(y^z) &= ({}^x y)^{({}^{(x^y)} z)},\\
{}^x({}^y z) &= {}^{({}^x y)}({}^{(x^y)} z).
\end{align}
\end{enumerate}
\end{lemma}

\begin{proof}
1 is just the definition of a $(X,R)$ being non-degenerate. By proposition 3.1.3, $(R_{21})^2 = \text{id}$ is equivalent to $(pR)^2 = \text{id}$, which we now see is equivalent to 3.1 and 3.2:
\[ (x,y) = (pR)^2(x,y) = pR({}^x y, x^y) = ({}^{({}^x y)}(x^y), {({}^x y)}^{(x^y)}). \]

We see that 3.3, 3.5 and 3.5 are equivalent to $R_{12} R_{13} R_{23} = R_{23} R_{13} R_{12}$ by calculating each side separately:
\begin{align*}
R_{12} R_{13} R_{23}(x,y,z) &= R_{12} R_{13} (x, y^z, {}^y z)\\
&= R_{12}(x^{({}^y z)}, y^z, {}^x ({}^y z))\\
&= ((x^{({}^y z)})^{(y^z)}, {}^{(x^{({}^y z)})}(y^z), {}^x ({}^y z)),\\
\\
R_{23} R_{13} R_{12}(x,y,z) &= R_{23} R_{13} (x^y, {}^x y, z)\\
&= R_{23} ((x^y)^z, {}^x y, {}^{(x^y)} z)\\
&= ((x^y)^z, ({}^x y)^{({}^{(x^y)} z)}, {}^{({}^x y)}({}^{(x^y)} z)).
\end{align*}
\end{proof}

\subsection{Cycle Sets}

\begin{definition}
A set $X$ with binary operation $\cdot$ is called a \emph{cycle set}, iff for all $x \in X$ the left multiplication $\sigma_x:y \mapsto x \cdot y$ is a bijection, and
\begin{equation}
(x \cdot y)\cdot(x \cdot z) = (y \cdot x)\cdot(y \cdot z)
\end{equation}
for all $x,y,z \in X$. Introducing the notation $y^x = \sigma_x^{-1}(y)$, we have
\begin{equation}
x \cdot y^x = (x\cdot y)^x = y.
\end{equation}
\end{definition}

The similarity in notation between $R(x,y) = (x^y, {}^x y)$ (or $R(x,y) = ({}^x y, x^y)$) for solutions to the YBE, and $x^y$ for the inverse multiplication in a cycle set, is no coincidence. Rump showed that defining ${}^x y = x^y \cdot y$ on a cycle set $(X,\cdot)$, we have a unitary set-theoretic solution to the QYBE $(X,R)$ with $R(x,y) = (x^y, {}^x y)$. In fact we have a bijection between cycle sets and \emph{left non-degenerate} solutions to the QYBE, which are solutions $(X,R)$ where $R(x,y) = (x^y, {}^x y)$, for which $x \mapsto x^y$ is a bijection. (So all non-degenerate solutions are left non-degenerate, but a left non-degenerate solution is not necessarily non-degenerate.)

In his paper \cite{ru1} Rump states this bijective correspondence but doesn't provide a complete proof. In the case of unitary solutions he proves that 3.3 and 3.6 are equivalent, but he merely states the fact that 3.4 and 3.5 also follow from 3.6. Therefore we provide our own proof.

\begin{theorem}[{\cite[Prop. 1]{ru1}}]
There is a bijective correspondence between cycle sets and left non-degenerate unitary set theoretic solutions to the QYBE.
\end{theorem}

\begin{proof}
Given a cycle set $(X, \cdot)$ define $R(x,y) = (x^y, {}^x y)$ where ${}^x y = x^y \cdot y$. We have that $x \mapsto x^y = \sigma_y^{-1}(x)$ is a bijection since $\sigma_y$ is a bijection, so $r$ is left non-degenerate.

We calculate that it satisfies 3.1 and 3.2.
\begin{align*}
{({}^x y)}^{(x^y)} &= {(x^y \cdot y)}^{(x^y)}\\
&= y,\\
\\
{}^{({}^x y)}(x^y) &= {({}^x y)}^{(x^y)} \cdot x^y\\
&= y \cdot x^y\\
&= x.
\end{align*}

We have left to prove that assuming the above conditions 3.6, 3.3, 3.4 and 3.5 are equivalent. First we show that 3.6 and 3.3 are equivalent:
\begin{align*}
(x^{({}^y z)})^{(y^z)} = (x^y)^z  &\iff  x^{({}^y z)} =  y^z \cdot (x^y)^z\\
&\iff  x = {}^y z \cdot(y^z \cdot (x^y)^z),
\end{align*}
then substitute $x = (y \cdot (z\cdot x')) = \sigma_y\sigma_z(x')$:
\begin{align*}
y \cdot (z \cdot x') &= {}^y z \cdot(y^z \cdot ((y \cdot (z \cdot x'))^y)^z)\\
& = {}^y z \cdot(y^z \cdot x')\\
& = (y^z \cdot z)\cdot (y^z \cdot x'),
\end{align*}
and substitute $y = z\cdot y' = \sigma_z(y')$:
\begin{align*}
(z \cdot y') \cdot (z \cdot x') &= ((z \cdot y')^z \cdot z)\cdot ((z \cdot y')^z \cdot x')\\
&= (y' \cdot z)\cdot (y' \cdot x').
\end{align*}
This is 3.6, and since $\sigma_y\sigma_z$ and $\sigma_z$ are bijections, our substitutions are invertible and therefore we have equivalence.

We now prove that 3.6 and 3.4 are equivalent when assuming the unitary condition. From 3.4, we have:
\begin{align*}
    {}^{(x^{({}^y z)})}(y^z) = ({}^{x}y)^{({}^{(x^y)}z)} &\iff ({}^{(x^y)}z) \cdot ({}^{(x^{({}^{y}z)})}(y^z)) = {}^{x}y,
\end{align*}
substitute $x =({}^{y}z)\cdot x_{1} =\sigma_{\sigma_{\sigma_z^{-1}(y)}(z)}(x_{1})$ on the right hand side equation, to get:
\begin{align*}
    (((({}^{y}z)\cdot x_{1})^{y})^{z}\cdot z)\cdot (({x_{1}}^{(y^{z})})\cdot (y^z)) = (({}^{y}z)\cdot x_{1})^y \cdot y,
\end{align*}
then substitute $y=z\cdot y_{1} = \sigma_z(y_{1})$:
\begin{align*}
    (((y_{1} \cdot z)\cdot x_{1})^{z \cdot y_{1}})^z)\cdot z) \cdot (x_{1}^{y_{1}} \cdot y_{1}) = ((y_{1} \cdot z) \cdot x_{1})^{z \cdot y_{1}}\cdot (z\cdot y_{1}),
\end{align*}
and substitute $x_{1}= ((z \cdot y_{1})\cdot(z\cdot x_{2}))^{(y_{1}\cdot z)} = \sigma_{\sigma_{y_1}(z)}^{-1} \sigma_{\sigma_z(y_1)} \sigma_z (x_2)$, which is equivalent to letting $(y_{1}\cdot z)\cdot x_{1} = (z\cdot y_{1})\cdot (z\cdot x_{2})$:
\begin{align*}
    (z\cdot x_{2})\cdot(z\cdot y_{1}) &= ((z\cdot x_{2})^z\cdot z)\cdot((((z\cdot y_{1})\cdot(z\cdot x_{2}))^{(y_{1}\cdot z)})^{y_{1}}\cdot y_{1})\\
    &= (x_{2}\cdot z)\cdot((((z\cdot y_{1})\cdot(z\cdot x_{2}))^{(y_{1}\cdot z)})^{y_{1}}\cdot y_{1}).
\end{align*}
Since $\sigma_x$ is a bijection, it follows that 3.4 holds iff:
    $$x_{2}=(((z\cdot y_{1})\cdot(z\cdot x_{2}))^{(y_{1}\cdot z)})^{y_{1}}.$$
Notice that: 
\begin{align*}
    (((y_{1}\cdot z)\cdot(y_{1}\cdot x_{2}))^{(y_{1}\cdot z)})^{y_{1}} &= (y_{1}\cdot x_{2})^{y_{1}}\\
    &= x_{2}.
\end{align*}
Therefore 3.4 holds iff:
\begin{align*}
    (((y_{1}\cdot z)\cdot(y_{1}\cdot x_{2}))^{(y_{1}\cdot z)})^{y_{1}} = (((z\cdot y_{1})\cdot(z\cdot x_{2}))^{(y_{1}\cdot z)})^{y_{1}},
\end{align*}
this is equivalent to:
\begin{align*}
    (y_{1}\cdot z)\cdot(y_{1}\cdot x_{2}) = (z\cdot y_{1})\cdot(z\cdot x_{2})
\end{align*}
because $\sigma_x$ is a bijection. Again, all of our substitutions were invertible, and hence 3.4 and 3.6 are equivalent. 
    
Finally we prove that assuming the unitary condition, 3.5 is equivalent to 3.6. From 3.5, we have:
\begin{align*}
    (x^{({}^{y}z)})\cdot ({}^{y}z) = ((^{x}y)^{(^{(x^{y})}z)})\cdot(^{(x^{y})}z),
\end{align*}
substitute $x=y\cdot x_{1} = \sigma_y(x_{1})$:
\begin{align*}
    ((y\cdot x_{1})^{(y^{z}\cdot z)})\cdot(^{y}z) = ((x_{1}\cdot y)^{(^{x_{1}}z)}),
\end{align*}
then substitute $y=z\cdot y_{1} = \sigma_z(y_{1})$:
\begin{align*}
    (((z\cdot y_{1} )\cdot x_{1})^{(y_{1}\cdot z)})\cdot (y_{1}\cdot z) = ((x_{1}\cdot (z\cdot y_{1}))^{(^{x_{1}}z)})\cdot (^{x_{1}}z),
\end{align*}
and finally substitute $x_{1} = z\cdot x_{2} = \sigma_z(x_{2})$:
\begin{align*}
    (((z\cdot y_{1} )\cdot (z\cdot x_{2}))^{(y_{1}\cdot z)})\cdot (y_{1}\cdot z) = (((z\cdot x_{2})\cdot (z\cdot y_{1}))^{(x_{2}\cdot z)})\cdot (x_{2}\cdot z).
\end{align*}
Notice this equation is very similar to 3.6. In fact if we assume 3.6 holds, in particular $(z\cdot y_{1})\cdot(z\cdot x_{2}) = (y_{1}\cdot z)\cdot(y_{1}\cdot x_{2})$, we get:
\begin{align*}
    (y_{1}\cdot x_{2})\cdot(y_{1}\cdot z) = ((z\cdot x_{2})\cdot (z\cdot y_{1})^{(x_{2}\cdot z)})\cdot (x_{2}\cdot z).
\end{align*}
If we assume again that 3.6 holds, in particular $(z\cdot x_{2})\cdot(z\cdot y_{1}) = (x_{2}\cdot z)\cdot(x_{2}\cdot y_{1})$, we get:
\begin{align*}
    (y_{1}\cdot x_{2})\cdot(y_{1}\cdot z) = (x_{2}\cdot y_{1})\cdot (x_{2}\cdot z), 
\end{align*}
which is 3.6. Now, since $\sigma_x$ is a bijection and again all our substitutions were invertible, it follows that 3.6 and 3.5 are equivalent.
\end{proof}

The second half of this proof actually gives us a (novel) corollary which will be very useful for checking when a map satisfying the unitary condition is a solution:

\begin{corollary}
Let $X$ be a set and $R:X^2 \to X^2$ - $r(x,y) = (x^y, {}^x y)$ - be a map satisfying the unitary condition. Then $(X,R)$ is a left non-degenerate unitary solution of the QYBE if $x \mapsto x^y$ is a bijection, and it satisfies one of 3.3, 3.4 or 3.5.
\end{corollary}

This can be restated as a lemma on whether an arbitrary involution is a solution to the YBE, which people studying non-degenerate, involutive solutions would find very useful to save time in their calculations:

\begin{lemma}
Let $X$ be a set and $r:X^2 \to X^2$ be given by $r(x,y) = (\lambda_x(y), \tau_y(x))$. Then $(X,r)$ is a solution to the YBE if and only if the following hold:
\begin{align*}
\tau_x \tau_y &= \tau_{\tau_x(y)}\tau_{\lambda_y(x)},\\
\lambda_{\tau_{\lambda_y(z)}(x)} \tau_z(y) &= \tau_{\lambda_{\tau_y(x)}(z)} \lambda_x(y),\\
\lambda_x \lambda_y &= \lambda_{\lambda_x(y)} \lambda_{\tau_y(x)}.
\end{align*}
If additionally $r$ is an involution, so $r^2 = \id$, then only one of the 3 conditions above needs to hold for $(X,r)$ to be a solution to the YBE.
\end{lemma}

\begin{proof}
Let $R(x,y) = (x^y, {}^x y) = (\tau_y(x),\lambda_x(y))$, so that by 3.1.3 $(X,pR) = (X,r)$ is a solution to the YBE if and only if $(X,R)$ is a solution to the QYBE. By 3.1.4 this is the case when $R$ satisfies 3.3, 3.4 and 3.5, which substituting $x^y = \tau_y(x)$ and ${}^x y = \lambda_x(y)$ give:
\begin{align*}
\tau_z (\tau_y(x)) = \tau_{\tau_z(y)}(\tau_{\lambda_y(z)}(x)),\\
\lambda_{\tau_{\lambda_y(z)}(x)} (\tau_z(y)) = \tau_{\lambda_{\tau_y(x)}(z)}( \lambda_x(y)),\\
\lambda_x (\lambda_y(z)) = \lambda_{\lambda_x(y)}( \lambda_{\tau_y(x)}(z)),
\end{align*}
which give the above. By the above corollary, if $R$ satisfies the unitary condition then only one of 3.3, 3.4 and 3.5 need hold, and so since $R$ is unitary if and only if $r$ is an involution (again by 3.1.3), if $r$ is an involution then only one of the above need hold.
\end{proof}

We can restrict the bijection of theorem 3.2.2 from left non-degenerate solutions to non-degenerate solutions by imposing an additional condition on the cycle sets.

\begin{definition}
We call a cycle set $X$ \emph{non-degenerate} iff the map $x \mapsto x \cdot x$ is a bijection. A binary operation $\odot:X^2 \to X$ will be called the \emph{dual} to the operation $\cdot$ if the following holds for all $x,y \in X$:
\begin{align}
(x \cdot y) \odot (y \cdot x) &= x,\\
(x \odot y) \cdot (y \odot x) &= x.
\end{align}
\end{definition}

\begin{proposition}[{\cite[Prop. 2]{ru1}}]
The following are equivalent:
\begin{enumerate}
\item $X$ is a non-degenerate cycle set,
\item $(X,R)$ is a non-degenerate solution to the QYBE,
\item There exists a dual to the operation $\cdot:X^2 \to X$.
\end{enumerate}
\end{proposition}

\begin{proof}
($1 \Rightarrow 2$) We need to show that $y \mapsto {}^x y$ is a bijection. Let $x,z \in X$. There exists a unique $y \in X$ such that $y \cdot y = (z \cdot z)^x$, so:
$$ z \cdot z = x \cdot (y \cdot y) = (y \cdot x^y) \cdot (y \cdot y) = (x^y \cdot y) \cdot (x^y \cdot y) = ({}^x y) \cdot ({}^x y). $$

Thus $z = {}^x y$, so since $y$ is uniquely defined, $y \mapsto {}^x y$ has a well defined inverse.

($2 \Rightarrow 3$) Let $x \mapsto y \odot x$ denote the inverse of $x \mapsto {}^y x$. We can immediately see 3.8 is satisfied since $y \mapsto y^x$ is a bijection:
$$ x = y \odot {}^y x = y \odot (y^x \cdot x) = (x \cdot y^x) \odot (y^x \cdot x). $$
Now, 3.1 gives us $x^y = {}^x y \odot x$, so substituting $y = x \odot z$ ($z$ is unique since $z \mapsto x \odot z$ is a bijection) we have $x^{x \odot z} = {}^x (x \odot z) \odot x = z \odot x$. Thus 3.9 is also satisfied:
$$ x = (x \odot z) \cdot x^{x \odot z} = (x \odot z) \cdot (z \odot x). $$

($3 \Rightarrow 1$) We show that $x \mapsto x \cdot x$ is a bijection by constructing its inverse. Letting $y = x$ in 3.8 and 3.9 we have:
$$ (x \cdot x) \odot (x \cdot x) = x = (x \odot x) \cdot (x \odot x), $$
so $x \mapsto x \odot x$ is inverse to $x \mapsto x \cdot x$.
\end{proof}

\begin{corollary}
There is a bijective correspondence between non-degenerate cycle sets and non-degenerate unitary solutions to the QYBE.
\end{corollary}

We now see that the dual operation on a non-degenerate cycle set allows us to define notions of a dual cycle set and therefore of a dual solution to the QYBE.

\begin{proposition}[{\cite[Prop. 2]{ru1}}]
Let $(X,\cdot)$ be a non-degenerate cycle set with dual operation $\odot$. Then $(X, \odot)$ is a non-degenerate cycle set.
\end{proposition}

\begin{proof}
Above we saw that ${}^x y \odot x = x^y$ and that $x \odot {}^x y = {}^x (x \odot y) = y$. We can derive that $\odot$ satisfies 3.6 from $(x^y, {}^x y)$ satisfying 3.5 (the calculation is dual to one from the proof of theorem 3.2.2, so it is given in less detail):
\begin{align*}
{}^{({}^x y)}({}^{(x^y)} z) = {}^x({}^y z) &\Rightarrow z = x^y \odot ({}^x y \odot {}^x ({}^y z))\\
&\Rightarrow y \odot (x \odot z') = x^y \odot ({}^x y \odot z')\\
&\Rightarrow y \odot (x \odot z') = ({}^x y \odot x) \odot ({}^x y \odot z')\\
&\Rightarrow (x \odot y') \odot (x \odot z') = (y' \odot x) \odot (y' \odot z').
\end{align*}

Thus $X^{\odot} := (X, \odot)$ is a cycle set, and it is non-degenerate since a dual to $\odot$ exists by definition.
\end{proof}

Since non-degenerate cycle sets correspond to unitary non-degenerate solutions, if $(X,R)$ is a solution where $R(x,y) = (x^y, {}^x y)$, then we can obtain a second solution $(X, R^{\odot})$ from $X^{\odot}$ where $R^{\odot}(x,y) = ({}^y x, y^x)$. We call this the \emph{dual solution} to $(X,R)$. By combining this with proposition 3.1.3, we discovered a lemma we could not find anywhere in the literature, which will make going back and forth between YBE and QYBE solutions even easier:

\begin{lemma}
Let $(X,r)$ be a non-degenerate, involutive solution to the YBE, $(Y,R)$ be a non-degenerate, unitary solution to the QYBE, and $p$ be the permutation $p(x,y) = (y,x)$. Then $(X,prp)$ and $(Y,Rp)$ are solutions to the YBE, and $(X,rp)$ and $(Y,pRp)$ are solutions to the QYBE.
\end{lemma}

\begin{proof}
Firstly, $(Y,pRp) = (Y,R^{\odot})$ as given above, so by proposition 3.1.3 $(Y,p^2Rp)\\ = (Y,Rp)$ is a non-degenerate, involutive solution to the YBE. Then by inverting proposition 3.1.3, $(X,pr)$ is a non-degenerate, unitary solution to the QYBE, and its dual is $(X,p^2rp) = (X,rp)$, so by once again applying 3.1.3, $(X,prp)$ is a non-degenerate, involutive solution to the YBE. 
\end{proof}

Since it is defined in the same way as $R^{\odot}$, we will denote $r^{\odot} = prp$, and say that $(X,r^{\odot})$ is the dual solution to $(X,r)$. The existence of the dual solution is \cite[Lem. 4.5]{ds}. Although the we used the fact that a QYBE solution $(X,R)$ is associated with a YBE solution $(X,pR)$ to prove this lemma, it turns out to be much more useful to associate $(X,R)$ with $(X,Rp)$, as we will see in chapter 4.

\subsection{Cycle Groups, Retractions and Linearity}

Let $G(X)$ denote the subgroup of $Sym(X)$ generated by the image of the map $\sigma:X \to Sym(X)$, given by $x \mapsto \sigma_x$. When thinking of the cycle set $X$ as a solution to the QYBE, $G(X)$ is often referred to as the \emph{permutation group} of the solution.

\noindent
\emph{Remark:} For a non-degenerate solution $r(x,y) = (\lambda_x(y),\tau_y(x))$ (to the YBE, not QYBE), the permutation group is defined as the subgroup generated by the $\lambda_x$s, which in the case of an involutive solution is the same as the subgroup generated by the $\tau_x$s (since $\tau_x(y) = \lambda_{\lambda_y(x)}^{-1}(y)$ by lemma 2.1.4). The bijections given in 3.1.3 and 3.2.2 give us $\tau_x = \sigma_x^{-1}$, so these clearly generate the same subgroup as the $\sigma_x$s.

Define a map $+:G(X) \times X \to G(X)$ by
\begin{equation}
(\pi + x)(y) = \pi(x)\cdot\pi(y)
\end{equation}
for $\pi \in G(X)$, $x,y \in X$. Taking another $\rho \in G(X)$ and $z \in X$, then we have
\begin{align*}
(\pi\rho + x)(y) &= \pi(\rho(x))\cdot\pi(\rho(y))\\
&= (\pi + \rho(x))(\rho(y)),\\
\\
((\pi + x) + y)(z) &= (\pi + x)(y)\cdot(\pi + x)(z)\\
&= (\pi(x)\cdot \pi(y))\cdot(\pi(x) \cdot \pi(z))\\
&= (\pi(y)\cdot\pi(x))\cdot(\pi(y) \cdot \pi(z))\\
&= ((\pi + y) + x)(z).
\end{align*}

This motivates the following definition:

\begin{definition}
Let $G$ be a group acting on a set $X$. The pair $(G,X)$ together with a map $+:G \times X \to G$ is called a \emph{cycle group} if for all $x,y \in X$, $\pi, \rho \in G$ we have:
\begin{align}
\pi\rho + x &= (\pi + \rho(x))\rho,\\
(\pi + x) + y &= (\pi + y) + x.
\end{align}
When the underlying set is clear from the context, we abbreviate $(G,X)$ to $G$.
\end{definition}

The conversation above shows that a cycle set $X$ naturally defines a cycle group $G(X)$. The next proposition shows the converse:

\begin{lemma}[{\cite[Prop. 4]{ru1}}]
Let $(G,X)$ be a cycle group. The operation $\cdot: X^2 \to X$ defined by
\begin{equation}
x \cdot y = (1 + x)(y)
\end{equation}
makes $X$ into a cycle set (where $1 \in G$ is the identity element).
\end{lemma}

\begin{proof}
Let $x,y,z \in X$:
\begin{align*}
(x \cdot y) \cdot (x \cdot z) &= (1 + x)(y) \cdot (1 + x)(z)\\
&= (1 + (1 + x)(y))((1 + x)(z))\\
&= (1(1+x) + y)(z)\\
&= (1(1+y) + x)(z)\\
&= (y\cdot x)\cdot(y\cdot z).
\end{align*}
\end{proof}

Recall that a set with a binary operation is called a \emph{monoid} if the operation is associative and has an identity element. Let $\N^{(X)}$ denote the \emph{free commutative monoid} generated by the set $X$: the operation will be denoted by $+$ and the identity by 0, so elements of $\N^{(X)}$ are finite linear combinations $\sum n_x x$ for $n_x \in \N$, where $n x = \sum_{i=1}^n x$.

\begin{proposition}[{\cite[Prop. 5]{ru1}}]
Every cycle group $(G,X)$ naturally extends to a cycle group $(G,\N^{(X)})$.
\end{proposition}

\begin{proof}
We extend the $G$-action to $\N^{(X)}$ by:
\begin{equation}
\pi\left(\sum n_x x\right) = \sum n_x \pi(x),
\end{equation}
and we extend the $+$ map by:
\begin{equation}
\pi + 0 = \pi, \quad \pi + (a + b) = (\pi + a) + b.
\end{equation}

3.12 follows immediately from 3.15 and the fact that $+$ is commutative. We prove 3.11 by induction on what we will call the \emph{length} of an element of $\N^{(X)}$, which we define as $\sum n_x$. The length 0 case is trivial, so suppose 3.11 holds for elements with length $\leq n$. Any element of length $n+1$ can be written as $y+x$ for $x \in X$ and $y \in \N^{(X)}$ where $y$ has length $n$. Thus we have:
\begin{align*}
\pi\rho + (y+x) &= (\pi\rho + y) + x\\
&= (\pi + \rho(y))\rho + x\\
&= ((\pi + \rho(y)) + \rho(x))\rho\\
&= (\pi + (\rho(y) + \rho(x))) \rho\\
&= (\pi + \rho(y + x))\rho.
\end{align*}
\end{proof}

In particular this means that any cycle set $X$ can be naturally extended to $\N^{(X)}$ such that $G(X) = G(\N^{(X)})$. The question which now arises is when can we extend a cycle set in the same way to the free abelian group $\Z^{(X)}$. 

\begin{theorem}[{\cite[Prop. 6]{ru1}}]
A cycle group $(G,X)$ admits an extension to $(G, \Z^{(X)})$ if and only if the underlying cycle set $X$ is non-degenerate.
\end{theorem}

\begin{proof}
Suppose an extension exists. Then we have:
\begin{align*}
x &= (1 + (-x + x))(x)\\
&= (1(1 - x) + x)(x)\\
&= (1 + (1 - x)(x))((1 - x)(x))\\
&= (1-x)(x)\cdot (1-x)(x),
\end{align*}
so since $1-x \in G$, $(1-x)(x) \in X$, meaning $x\mapsto (1-x)(x)$ is inverse to $x \mapsto x\cdot x$ and therefore $X$ is non-degenerate.

Conversely suppose $(X,\cdot)$ is non-degenerate and let $\odot$ be the dual operation. We extend the $G$-action by 3.14, and the $+$ map by 3.15 in addition to:
\begin{equation}
(\pi - x)(y) = \pi(y)^{\pi(x)\odot\pi(x)}.
\end{equation}

First we need to show that $(\pi - x) + x = \pi = (\pi + x) - x$ so that 3.10 and 3.16 are mutually inverse and therefore 3.15 holds when $a \in \Z^{(X)}$ and $b \in \pm X$. Using 3.13 and 3.8 we have:
\begin{align*}
((\pi + x) - x)(y) &= (\pi + x)(y)^{(\pi + x)(x)\odot (\pi + x)(x)}\\
&= (\pi(x)\cdot\pi(y))^{(\pi(x)\cdot\pi(x))\odot(\pi(x)\cdot\pi(x))}\\
&= (\pi(x)\cdot\pi(y))^{\pi(x)}\\
&= \pi(y),
\end{align*}
and 3.9 gives us $\pi(x) = (\pi(x) \odot \pi(x))\cdot(\pi(x) \odot \pi(x))$, so:
\begin{align*}
(\pi - x)(x) &= \pi(x)^{\pi(x)\odot\pi(x)}\\
&= ((\pi(x) \odot \pi(x))\cdot(\pi(x) \odot \pi(x)))^{\pi(x)\odot\pi(x)}\\
&= \pi(x) \odot \pi(x).
\end{align*}
Hence:
\begin{align*}
((\pi - x) + x)(y) &= (\pi - x)(x)\cdot(\pi - x)(y)\\
&= (\pi(x) \odot \pi(x))\cdot \pi(y)^{\pi(x)\odot\pi(x)}\\
&= \pi(y).
\end{align*}

From this it follows that $\pi + t$ is well defined for $t \in \Z^{(X)}$, since $t$ is a sum of positive and negative elements of $X$, so $\pi + t$ can be obtained by sequentially adding (3.10) and subtracting (3.16) elements of $X$. Thus since 3.15 holds, it follows that 3.12 also holds.

We can extend the inductive argument in the proof of proposition 3.3.3 to obtain 3.11. Elements of $\Z^{(X)}$ are finite linear combinations $\sum n_x x$ for $n_x \in \Z$, so we redefine length to be $\sum |n_x|$. Again the length 0 case is trivial, so suppose 3.11 holds for elements of $\Z^{(X)}$ of length $\leq n$. Any element of length $n+1$ can be written as $y + x$ for $x,y \in \Z^{(X)}$ where $x$ has length 1 (i.e. $x = \pm x'$ for $x' \in X$) and $y$ has length $n$. Since we have already shown 3.12 holds, the calculation in the proof of 3.3.3 remains valid here, so by induction 3.11 holds on all of $\Z^{(X)}$.
\end{proof}

Lemma 3.3.2 says that extending $(G,X)$ to $(G,\Z^{(X)})$ means that $\Z^{(X)}$ is a cycle set, which we will call the \emph{linear extension} of $X$. This has certain properties which lead to the definition of an important class of cycle sets with a compatible abelian group structure. 

\begin{definition}
A cycle set $A$ is called a \emph{linear cycle set}, iff it is an abelian group such that for all $a,b,c \in A$:
\begin{align}
a \cdot (b+c) &= a \cdot b + a \cdot c,\\
(a+b)\cdot c &= (a \cdot b)\cdot(a\cdot c).
\end{align}
\end{definition}

We verify that the linear extension of a cycle set satisfies this definition:

\begin{proposition}
The linear extension $\Z^{(X)}$ of a non-degenerate cycle set $X$ is a linear cycle set.
\end{proposition}

\begin{proof}
3.17 simply follows from 3.14:
\begin{align*}
a \cdot (b + c) &= \sigma_a(b+c)\\
&= \sigma_a(b) + \sigma_a(c)\\
&= a \cdot b + a \cdot c,
\end{align*}
and then 3.18 follows from 3.11, 3.13 and 3.15:
\begin{align*}
(a + b) \cdot c &= (1 + (a + b))(c)\\
&= ((1+a) + b)(c)\\
&= (1 + (1+a)(b))((1+a)(c))\\
&= (1 + a\cdot b)(a \cdot c)\\
&= (a \cdot b) \cdot (a \cdot c).
\end{align*}
\end{proof}

Here are some useful properties of linear cycle sets which will make our calculations easier:

\begin{lemma}
Let $A$ be a linear cycle set, and $a,b \in A$. Then we have the following:
\begin{equation}
a \cdot(-b) = -(a\cdot b),
\end{equation}
\begin{equation}
a \cdot 0 = 0, \quad 0 \cdot a = a.
\end{equation}
\end{lemma}

\begin{proof}
For (3.19), first substitute $c = -2b$ into 3.17 to get $a \cdot (-b) = a \cdot b + a \cdot (-2b)$, then applying 3.17 again to $a \cdot (-2b) = a \cdot(-b - b)$ we get $a \cdot (-b) = a\cdot b + a \cdot(-b) + a\cdot(-b)$, so subtracting $a\cdot b + a \cdot(-b)$ from both sides we get $-(a\cdot b) = a \cdot (-b)$.

For (3.20), substituting $c = -b$ into 3.17 we have $a \cdot 0 = a\cdot b + a\cdot (-b) = a\cdot b - a\cdot b = 0$. Now substituting $a = b = 0$ into 3.18 we have $0 \cdot c = (0\cdot 0)\cdot(0 \cdot c) = 0 \cdot (0\cdot c)$, so since the map $c \mapsto 0\cdot c$ is a bijection we have $c = 0\cdot c$.
\end{proof}

\begin{proposition}[{\cite[Prop. 9]{ru1}}]
Let $A$ be a linear cycle set. Then $A$ is non-degenerate and the map $\sigma:A \to G(A)$ is a surjection.
\end{proposition}

\begin{proof}
By 3.18 and 3.20 we have
\begin{align*}
a &= 0 \cdot a\\
&= (-a + a) \cdot a\\
&= ((-a)\cdot a)\cdot((-a)\cdot a),
\end{align*}
meaning that $a \mapsto (-a) \cdot a$ is inverse to $a \mapsto a \cdot a$, therefore $A$ is non-degenerate.

Since $G(A)$ is the group generated by $\sigma(A)$, for surjectivity of $\sigma$ we need to check that $\sigma_a^{-1} \in \sigma(A)$ for all $a \in A$. By 3.11 we have:
\begin{equation}
 1 + a + b = 1(1+a) + b = (1 + (1+a)(b))(1+a) = (1 + a \cdot b)(1+a). \tag{*}
\end{equation}

Thus if we let $b = -a$ then 3.19 tells us that:
\[ (1 - a\cdot a)(1+a) = 1, \]
so $\sigma_{-a\cdot a} = \sigma_a^{-1}$.
\end{proof}

Recall that given any map between sets $f:X \to Y$, we can define an equivalence relation on $X$ by $x \sim y \iff f(x) = f(y)$ for $x,y \in X$. If we have a binary operation $\cdot$ on $X$, then a \emph{congruence relation} on $X$ is an equivalence relation $\sim$ such that:
\[ x \sim x', \ y \sim y' \Rightarrow x \cdot y = x' \cdot y'. \]

\begin{lemma}[{\cite[Lem. 2]{ru1}}]
Let $X$ be a cycle set. Then $x \sim y \iff \sigma_x = \sigma_y$ is a congruence relation on $X$.
\end{lemma}

\begin{proof}
Let $x,x',y,y' \in X$ be such that $\sigma_x = \sigma_{x'}$ and $\sigma_y = \sigma_{y'}$. Since this means that $x \cdot z = x' \cdot z$ and $y \cdot z = y' \cdot z$ for all $z \in X$, we have:
\begin{align*}
(x \cdot y) \cdot z &= (x \cdot y)\cdot(x \cdot z^x)\\
&= (y \cdot x)\cdot(y \cdot z^x)\\
&= (y' \cdot x)\cdot(y' \cdot z^x)\\
&= (x \cdot y')\cdot(x \cdot z^x)\\
&= (x' \cdot y') \cdot z,
\end{align*}
meaning $\sigma_{x \cdot y} = \sigma_{x' \cdot y'}$.
\end{proof}

A consequence of this lemma is that for any cycle set $X$ we can define a binary operation $\cdot$ on $\sigma(X) \subseteq G(X)$ by:
\begin{equation}
\sigma_x \cdot \sigma_y = \sigma_{x \cdot y}.
\end{equation}
If $\sigma(X)$ is a cycle set, we can induce a binary operation on $\sigma^2(X) = \sigma(\sigma(X)) \subseteq G(\sigma(X))$, and so on.

\noindent
\emph{Remark:} A map between cycle sets $f:X \to Y$ satisfying 3.21 is called a \emph{morphism} of cycle sets, and if it is also a bijection we call it an \emph{isomorphism} and write $X \cong Y$.

\begin{definition}
Let $X$ be a cycle set. We call $\sigma(X)$ the \emph{retraction} of $X$, and $\sigma^n(X)$ (if it exists) the \emph{$n$th retraction} of $X$. If there exists $m$ such that $\sigma^m(X) = 1$ (the trivial group), we call $X$ \emph{fully retractable}. If $|X| > 1$ and $\sigma:X \mapsto G(X)$ is injective, we say that $X$ is \emph{irretractable} and we have $X \cong \sigma(X)$.
\end{definition}

By the remark above, whenever $\sigma(X)$ is a cycle set we have a natural surjective morphism $\sigma:X \to \sigma(X)$. We now give two results, the proofs of which are slightly beyond our scope, but which are both very useful:

\begin{proposition}[{\cite[Prop. 10]{ru1}}]
A cycle set $X$ is non-degenerate if and only if its retraction $\sigma(X)$ is a non-degenerate cycle set.
\end{proposition}

\begin{proof}
See \cite{ru1}, page 53.
\end{proof}

\begin{theorem}[{\cite[Thm. 2]{ru1}}]
Every finite cycle set is non-degenerate.
\end{theorem}

\begin{proof}
See \cite{ru1}, page 53.
\end{proof}

Proposition 3.3.11 combined with 3.3.8 tells us that the retraction of a linear cycle set is a non-degenerate cycle set. When considering the linear extension of a cycle set, this retraction is worthy of its own notation.

\begin{definition}
Let $X$ be a non-degenerate cycle set. We denote the retraction of its linear extension $\Z^{(X)}$ by $A(X)$. Proposition 3.3.8 tells us that:
\[ A(X) = \Z^{(X)}/\sigma^{-1}(1). \]
\end{definition}

The composition $X \hookrightarrow \Z^{(X)} \to A(X)$ gives a natural morphism $\rho:X \to A(X)$ whose image can be identified with the retraction $\sigma(X)$. Thus $\rho$ is injective if and only if $X$ is irretractable.

A consequence of proposition 3.5.7, which we will see soon, is that the retraction of a linear cycle set is not only a non-degenerate cycle set, but a linear cycle set.

\subsection{Linear Cycle Sets are Right Braces}

In this section we will see that linear cycle sets are equivalent to right braces. To prove this, it will be easier if we think about linear cycle sets in terms of their inverse operation $b^a$ rather than the operation $a \cdot b$. To formulate a definition of a linear cycle set in terms of $a^b$ we define some new notation.

\begin{definition}
Let $A$ be a linear cycle set. We define the \emph{adjoint multiplication} $\circ$ on $A$ by:
\begin{align}
   a \circ b = a^b + b. 
\end{align}
\end{definition}

We will see that this corresponds exactly to the $\circ$ operation in a brace, but for now it is simply a convenient abbreviation for giving an equivalent formulation of definition 3.3.5.

\begin{proposition}[{\cite[Prop. 2]{ru2}}]
Let $(A,+)$ be an abelian group, and $\cdot$ be a binary operation on $A$ such that  $\sigma_a: b \mapsto a \cdot b$ is invertible for all $a \in A$. We write $a^b = \sigma_b^{-1}(a)$, and $a \circ b = a^b + b$ as before. Then $(A,+,\cdot)$ is a linear cycle set iff for $a,b,c \in A$:
\begin{align}
(a + b)^c &= a^c + b^c,\\
(a^b)^c &= a^{b\circ c}.
\end{align}
\end{proposition}

\begin{proof}
By 3.7, 3.17 and 3.23 are equivalent since:
\begin{align*}
\text{assuming 3.17: }& a^c + b^c = (c\cdot(a^c + b^c))^c = (c\cdot a^c + c\cdot b^c)^c = (a+b)^c,\\
\text{assuming 3.23: }& a \cdot(b+c) = a\cdot(a\cdot b^a + a\cdot c^a) = a\cdot(a\cdot b + a\cdot c)^a = a\cdot b + a\cdot a.
\end{align*}

Likewise we see that 3.18 and 3.24 are equivalent:
\begin{align*}
\text{assuming 3.18: }& (c+b^c)\cdot (a^b)^c = (c\cdot b^c)\cdot (c \cdot (a^b)^c) = b\cdot a^b = a\\
&\Rightarrow (a^b)^c = ((c + b^c) \cdot (a^b)^c)^{c+b^c} = a^{c+b^c} = a^{b\circ c},\\
\text{assuming 3.24: }& c =(a \cdot b) (a\cdot (c^{a\cdot b})^a) = (a \cdot b) \cdot (a \cdot c^{(a \cdot b)^a + a}) = (a\cdot b) \cdot (a \cdot c^{b + a})\\
&\Rightarrow (a + b)\cdot c = (a\cdot b) \cdot (a \cdot ((a+b)\cdot c^{a+b})) = (a\cdot b) \cdot (a \cdot c).
\end{align*}
\end{proof}

Now we are ready to show the connection between linear cycle sets and right braces:

\begin{theorem}[{\cite[Prop. 5]{ru2}}]
Let $(A,+,\cdot)$ be a linear cycle set and $\circ$ be the adjoint multiplication on $A$. Then $(A,+,\circ)$ is a right brace.
\end{theorem}

\begin{proof}
First we see that $(A,\circ)$ is a group. 0 is the identity in a brace, so we verify that this is the case. Using 3.7 and 3.22:
\begin{align*}
0^a = (0 + 0)^a = 0^a + 0^a &\Rightarrow 0 \circ a = 0^a + a = 0 + a = a,\\
a^0 = 0\circ a^0 = a &\Rightarrow a \circ 0 = a^0 + 0 = a.
\end{align*}

Now suppose $a^{-1}$ is such that $(a^{-1})\circ a = 0$. We can deduce:
\[ (a^{-1})^a + a = 0 \Rightarrow (a^{-1})^a = -a. \] 
$\sigma_a$ is invertible for all $a \in A$, hence: 
\begin{align*}
    (a^{-1})^{a} = \sigma_a^{-1}(a^{-1}) &\Rightarrow a\cdot (a^{-1})^{a} =  \sigma_a(\sigma_a^{-1}(a^{-1})) = a^{-1}\\
&\Rightarrow a^{-1} = a\cdot (a^{-1})^a = a\cdot(-a)
\end{align*}
giving us existence of the left inverse. By 3.19 we have $a^{-1} = -(a\cdot a)$ giving us uniqueness since by proposition 3.3.8 $A$ is non-degenerate. It remains to show that $a^{-1}$ is also a right inverse:
\begin{align*}
a = a^0 = a^{-(a\cdot a) \circ a} = (a^{-(a\cdot a)})^a &\Rightarrow a^{-(a\cdot a)} = a \cdot (a^{-(a\cdot a)})^a = a\cdot a\\
&\Rightarrow a \circ -(a\cdot a) = a^{-(a\cdot a)} - a\cdot a = a\cdot a - a\cdot a = 0.
\end{align*}

Associativity follows from 3.24, since:
\begin{align*}
& d^{(a\circ b)\circ c} = [d^{a\circ b}]^c = [(d^a)^b]^c = (d^a)^{b\circ c} = d^{a\circ(b\circ c)}\\
&\Rightarrow (a\circ b)\circ c = a\circ(b\circ c)
\end{align*} 

Thus to show that $(A,+,\circ)$ is a right brace, we need to verify the brace distributivity: $(a + b)\circ c + c = a\circ c + b\circ c$. Indeed we have:
\begin{align*}
(a + b)\circ c + c &= (a + b)^c + c + c\\
&= a^c + c + b^c + c\\
&= a\circ c + b\circ c.
\end{align*}
\end{proof}

So every linear cycle set has an associated right brace. In fact this goes both ways.

\begin{theorem}[{\cite[Prop. 5]{ru2}}]
The map from linear cycle sets to right braces $(A,+,\cdot) \mapsto (A,+,\circ)$ given in theorem 3.4.3 is a bijection.
\end{theorem}

\begin{proof}
Let $(B,+,\circ)$ be a right brace, and $\lambda^{op}_a(b) = b \circ a - a$ (as in definition 2.9, except with the opposite multiplication). We claim that the linear cycle set which maps to $(B,+,\circ)$ is $(B,+,\cdot)$ where:
\[ a \cdot b = (\lambda^{op}_a)^{-1}(b). \]

We need to show that $(B,+,\cdot)$ is in fact a linear cycle set, and that the brace multiplication coincides with the adjoint multiplication.

We follow proposition 3.5: $a^b = ((\lambda^{op}_b)^{-1})^{-1}(a) = \lambda^{op}_b(a) = a \circ b - b$. Then clearly we have $a^b + b = a \circ b$, so the brace multiplication and adjoint multiplication coincide as required. It remains to show then, that $a^b$ and $a \circ b$ satisfy 3.23 and 3.24.

\begin{align*}
\text{for 3.23: }& (a + b)^c = (a+b)\circ c - c = a \circ c + b \circ c - 2c = a^c + b^c,\\
\text{for 3.24: }&  (a^b)^c = (a \circ b - b) \circ c - c = a \circ b \circ c - b \circ c = a^{b \circ c}.
\end{align*}
(The second line uses 2.11).
\end{proof}

\subsection{Socle Series of a Brace}

In this section we will define the \emph{socle series} of a brace, and see how it allows us to think about retractability in terms of braces and ideals of braces. Previously we have defined braces in terms of the adjoint multiplication and only occasionally used the ring multiplication as a tool for proofs, but for the definitions in this chapter the ring multiplication is more convenient. (We abbreviate $a*b$ to $ab$ since there is no risk of confusion here.) We begin by reformulating definition 2.2.1.

\begin{proposition}
Let $(A,+)$ be an abelian group with multiplication $*$ (we will write $ab$ for $a*b$). Then $(A,+,\circ)$ is a left [right] brace where $a \circ b = a + b + ab$ if and only if for all $a,b,c \in A$, $x \mapsto xa + x$ [$x \mapsto ax + x$] is bijective and:
\begin{align}
(a + b)c = ac + bc, &\quad [\, a(b+c) = ab + ac, \,]\\
a(b + c + bc) &= ab + ac + (ab)c.
\end{align}
\end{proposition}

\begin{proof}
First we see that if 3.25 holds then 3.26 is equivalent to associativity of $\circ$:
\begin{align*}
a(b + c + bc) = ab + ac + (ab)c &\iff a(b\circ c) = ab + (a + ab)c\\
&\iff a + b \circ c  + a(b\circ c) = a + b \circ c + ab + (a + ab)c\\
&\iff a \circ (b \circ c) = a + b + c + bc + ab + (a + ab)c\\
&\iff a \circ (b \circ c) = (a + b + ab) + c + (a + b + ab)c\\
&\iff a \circ (b \circ c) = (a \circ b) \circ c,
\end{align*}
so the circle operation is associative. 3.25 instantly gives us that 0 is the identity. Given $a \in A$ we need to find $b \in A$ such that $b \circ a = 0$. Let $-b$ be the image of $a$ under the inverse of $x \mapsto xa + x$, so $(-b)a - b = a$. 3.25 implies that $(-b)a = -ba$, so $b + a + ba = 0$ as required. Thus $(A,\circ)$ is a group. Now we show that 2.5 is equivalent to 3.25:
\begin{align*}
&(a+b)\circ c + c = a\circ c + b\circ c\\
\iff& (a + b) + c + (a+b)c + c =  a + c + ac + b + c + b*c\\
\iff& (a+b)c = ac + bc.
\end{align*}

We have shown that if the proposition is satisfied we have a brace, and that the a brace satisfies 3.25 and 3.26, so it remains to show that $x \mapsto xa + x$ is a bijection. Let $A$ be a brace and let $a \in A$. Since $(A,+)$ and $(A,\circ)$ are groups, the operations $x \mapsto x \circ a$ and $x \mapsto x - a$ are bijections so their composition $x \mapsto (x \circ a) - a = x + a + xa - a = xa + x$ is a bijection.
\end{proof}

We now give two definitions for braces that will be familiar from ring theory:

\begin{definition}
Let $A$ be a right brace. A subset $B \subseteq A$ is a \emph{sub-brace} of $A$ if it is a subgroup of $(A,+)$ and $(A,\circ)$. A sub-brace $I \subseteq A$ is an \emph{ideal} of $A$ if $ax\in I$ and $xa\in I$ whenever $a\in A$ and $x\in I$. 
\end{definition}
One consequence of 3.22 is
\begin{align*}\label{}
    a^b = a\circ b -b = a+b+ab-b,
\end{align*}
which gives us
\begin{align}
 a^b = ab+a,   
\end{align}
this will be a useful equation for the following results in this section.
\begin{definition}
Let $A$ be a brace. An abelian group $M$ together with a right operation $M \times A \rightarrow M$ will be called an \emph{$A$-module}.
If 
\begin{gather*}
    (x+y)a = xa +ya,\\
    x(a\circ b) = (xa)b+xa+xb,\\
    x0 = 0
\end{gather*}
holds for all $x,y \in M$ and $a, b\in A$.
\end{definition}
Note that every $A$-module can be regarded as a right $A^{\circ}$-module by
\begin{align*}
    M\times A^{\circ} \rightarrow M\\
    (x,a) \mapsto xa+x,
\end{align*}
and vice versa.
\begin{definition}
We define \emph{socle series} of $A$ by: 
\begin{align*}
&Soc_{0}(A):=0\\
    &Soc_{n+1}(A):=\{x\in A \,|\, \forall a\in A : ax\in Soc_{n}(A)\},
\end{align*}
for all $n\in \mathbb{N}$. It is common to write $Soc(A)$ instead of $Soc_{1}(A)$, and call this the \emph{socle} of $A$:
\[ Soc(A) = \{ x \in A : \forall a \in A, ax = 0\}. \]
\end{definition}
\begin{definition}
If $I$ is an ideal and $A$ a brace. Then $A/I$ is a \emph{factor brace}.
\end{definition}

We will quickly justify that factor braces are well defined. Addition is well defined as $(I,+)$ is a normal subgroup of $(A,+)$, so $(A/I,+)$ is just the quotient group. For $a_{1}, a_{2}, b_{1}, b_{2}\in A$ let $a_{1}-a_{2}\in I$ and $b_{1}-b_{2}\in I$. We need to verify that $a_1 b_1 - a_2 b_2 \in I$, or equivalently that $a_1b_1 + I = a_2b_2 + I$, so that multiplication is well defined.

Since $ax \in I$ for $x \in I$, $I$ is an $A$-submodule and therefore an $A^{\circ}$-submodule. Hence for any $x\in I$ and any $c\in A$ there exists a unique $y\in I$ such that $x=yc+y$. Therefore if $a\in A$
\begin{align*}
    a(x+c) -ac &= a(yc+y+c) -ac\\
    &= (ay)c + ay+ ac -ac\\
    &= (ay)c+ay.
\end{align*}
Hence $a(x+c)-ac\in I$ for all $a, c\in A$ and $x\in I$. Using this fact,we have:
\begin{align*}
a_1b_2 - a_1b_1 = a_1((b_2 - b_1) + b_1) - a_1b_1 &\in I,\\
a_2b_2 - a_1b_2 = (a_2 - a_1)b_2 &\in I,
\end{align*}
so $a_1b_1 + I = a_1b_2 + I = a_2b_2 + I$ as required.

\begin{proposition}
$c\in Soc(A) \Longleftrightarrow a(b+c)=ab$, for all $a,b \in A$.
\end{proposition}
\begin{proof}
($\Rightarrow$) Let $c\in Soc(A)$, then 
\begin{align*}
    a(b+c) &= a(bc+b+c)\\
    &= (ab)c+ab+ac\\
    &= ab.
\end{align*}\\
($\Leftarrow$) Now assume that $a(b+c) = ab$ for all $a\in A$ and any $b\in A$. By proposition 3.5.1 we have that the map $d\mapsto db+d$ are bijective maps. Thus $c=db+d$ for a unique $d\in A$. Thus
\begin{align*}
    ab&= a(b+c)\\
    &= a(b+db +d)\\
    &= (ad)b +ad+ab,
\end{align*}
for all $a\in A$. Hence $(ad)b + ad= 0$. By theorem 3.4.3 we have that $(A, \cdot)$ is a cycle set, hence $(ad)b + ad= 0$ is equivalent to $(ad)^{b} =\sigma_b^{-1}(ad) = 0$. Left multiplication in the cycle set is a bijection, therefore
\begin{align*}
    ad&= \sigma_{b}(\sigma_{b}^{-1}(ad))\\
    &=\sigma_{b}(0)\\
    &= b\cdot 0\\
    &= 0,
\end{align*}
for all $a,b\in A$. The last equality is justified using lemma 3.3.7.
Hence $d\in Soc(A)$. Therefore
\begin{align*}
    0 &= ac\\
    &= a(db+d)\\
    &= a(db),
\end{align*}
thus, $d\in Soc(A) \Rightarrow db\in Soc(A)$ for any $b\in A$. Using the right implication proved above, we obtain
\begin{align*}
    ac &= a(db+d)\\
    &= a(db)\\
    &=0,
\end{align*}
meaning that $c\in Soc(A)$.
\end{proof}
In the proposition above we have proved that if $c\in Soc(A) \Rightarrow cb \in Soc(A)$ for any $b \in A$. This is going to be useful for the next proposition.
\begin{proposition}[{\cite[Prop. 7]{ru2}}]
Each $Soc_n(A)$ is an ideal of $A$, and the factor brace $A/Soc(A)$ is isomorphic to the retraction $\sigma(A)$.
\end{proposition}
\begin{proof}
We need to prove that $A/Soc(A)$ is a factor brace, but first in order to do this we must show that $Soc(A)$ is an ideal. Let $a\in A$ and $b, c, d\in Soc(A)$. We first show that $Soc(A)$ is a subgroup of $(A,+)$. By Proposition 3.5.6, we have 
\begin{align*}
    a(b+c) &= ab\\
    &= 0,
\end{align*}
since $b\in Soc(A)$. Hence $b+c\in Soc(A)$. If $c\in Soc(A)$, then clearly $-c\in Soc(A)$ too. It is associative because $(A,+)$ is associative and $Soc(A)\subseteq A$. The identity in $(Soc(A), +)$ is $0$, the same as in $(A, +)$. Therefore $Soc(A)$ is a subgroup of $(A,+)$.

We now prove that $Soc(A)$ is a subgroup of $(A, \circ)$. Since $Soc(A) \subseteq A$, we have 
\begin{align*}
    a(b\circ c) &= a(bc +b+ c)\\
    &= (ab)c + ab + ac\\
    &= 0.
\end{align*}
Therefore $b\circ c\in Soc(A)$. Let $c^{-1}$ be the inverse of $c$ in $(A,\circ)$, then
\begin{align*}
    0 &= c^{-1}\circ c\\
    &= c^{-1}c +c^{-1}+ c\\
    &= c^{-1} +c,
\end{align*}
hence $c^{-1} = -c\in A$. Therefore clearly $c^{-1}\in Soc(A)$. Now, $(Soc(A), \circ)$ is associative because $Soc(A) \subseteq A$ and $(A, \circ)$ is associative. The identity in $(Soc(A), \circ)$ is $0$ again, the same as in $(A, \circ)$. Meaning that $Soc(A)$ is a subgroup of $(A, \circ)$.

Notice that if $a\in A$ and $c\in Soc(A)$, then $ac = 0 \in Soc(A)$. By Proposition 3.5.6, we also have that $ca\in Soc(A)$. Hence $Soc(A)$ is an ideal, thus by induction one can observe that the socle series consists of ideals. Since $Soc(A)$ is an ideal, by definition 3.5.5 it follows that $A/Soc(A)$ is a well-defined factor brace.

Now, notice that if two elements $a,b\in A$ are mapped to the same element in $\sigma(A)$, so $\sigma_a = \sigma_b$, then also $\sigma_a^{-1} = \sigma_b^{-1}$. Thus we have
\begin{align*}
    \sigma_a^{-1}(x) = \sigma_b^{-1}(x)
    &\iff x^a = x^b\\
    &\iff xa+x = xb+ x\\
    &\iff xa = xb\\
    &\iff x(a +Soc(A)) = x(b+ Soc(A))\\
    &\iff a+Soc(A) = b+Soc(A),
\end{align*}
using 3.27 and proposition 3.5.6. Hence two elements from $A$ are mapped to the same element in $\sigma(A)$ if and only if they are mapped to the same element in $A/Soc(A)$.
\end{proof}

As mentioned earlier, a consequence of this result is that the retraction of a linear cycle set is also a linear cycle set, since it is a brace.

\section{Left Braces}

We can now now use our collection of results on cycle sets and right braces to prove key results on left braces. First we use the equivalences between cycle sets and solutions, and linear cycle sets and right braces, to prove theorem 2.2.8. We then go on to prove that the adjoint group of the linear extension $\Z^{(X)}$ of a cycle set $X$ is isomorphic to the \emph{structure group} $G_X$ of the solution arising from $X$.

\subsection{Left Braces Give Solutions}

Recall theorem 2.2.8 at the end of section 2, which stated without proof that a left brace $B$ gives rise to a solution to the YBE $(B,r)$:
\begin{equation}
r(a,b) = \left(\lambda_a(b), \lambda_{\lambda_{a}(b)}^{-1}(a)\right),
\end{equation}
where the \emph{lambda function} $\lambda_a$ for each $a \in B$ is given by $\lambda_a(b) = a \circ b - a$.

We have now developed the tools to prove this using cycle sets:

\begin{proof}
Let $B^{op}$ be the opposite right brace to our left brace $B$. Then $\lambda_a$ in terms of the opposite multiplication is given by $\lambda_a(b) = a \circ b - a = b \circ^{op} a - a = \lambda_a^{op}(b)$. Thus by theorem 3.4.4 we have a linear cycle set $(B, +, \cdot)$, where $+$ is inherited from the left brace and $a \cdot b = \lambda_a^{-1}(b)$.

Since linear cycle sets are non-degenerate cycle sets by proposition 3.3.9, $B$ has an associated non-degenerate, unitary solution $R(a,b) = (a^b, {}^a b)$ to the QYBE. By non-degeneracy and lemma 3.2.9, $(B,r)$ is a solution to the YBE where $r = Rp$ (and $p(x,y) = (y,x)$):
\[
r(a,b) = (b^a, {}^b a) = (b^a, b^a \cdot a),
\]
where $b \mapsto b^a$ is inverse to the map $b \mapsto a \cdot b$. By definition of $\cdot$ we have $b^a = \lambda_a(b)$. Thus since $b^a \cdot a = \lambda_{b^a}^{-1}(a) = \lambda_{\lambda_a(b)}^{-1}(a)$, $r$ is the map given in 4.1.
\end{proof}

This theorem characterises one of the two major ways that brace theory is used to study the Yang-Baxter equation: finding solutions. Many techniques have been developed to construct braces, and to assess whether braces exists with certain properties. What theorem 2.2.8 tells us is that any time you construct a brace, you get a set theoretic non-degenerate involutive solution for free. Braces are useful for more than just finding solutions however, as we will see in the next subsection.

\subsection{The Structure Group}

Before Rump's work on cycle sets and braces, the main algebraic structure associated to a set theoretic solution $(X,r)$ was its \emph{structure group}.

\begin{definition}
Let $(X,r)$ be a solution to the YBE. The structure group $G_X$ of $(X,r)$, is the group generated by $X$ subject to the relations:
\begin{equation}
xy = wz \quad \text{when} \quad r(x,y) = (w,z). 
\end{equation}
\end{definition}

Introduced by Etingof, Schedler and Soloviev \cite{ess}, the structure group captures all of the behaviour of the solution at the expense of being very large and complex. We will give some key properties of the structure group of non-degenerate involutive solutions, then go on to see how the theory of braces and cycle sets can help us to understand this complicated group.

\begin{proposition}[{\cite[Prop. 2.1]{ess}}]
Let $(X,r)$ be a non-degenerate solution to the YBE with $r(x,y) = (\lambda_x(y), \tau_y(x))$. Then we have
\begin{enumerate}
\item[(i)] The assignment $x \mapsto \lambda_x$ extends to a left action of $G_X$ on $X$,
\item[(ii)] The assigmnent $x \mapsto \tau_x$ extends to a right action of $G_X$ on $X$.
\end{enumerate}
\end{proposition}

\begin{proof}
What this is telling us is that the map $G_X \times X \to X$ given by $(x_1 \cdots x_n, y) \mapsto \lambda_{x_1}\cdots \lambda_{x_n}(y)$ is a group action (and similarly $(x_1 \cdots x_n, y) \mapsto \tau_{x_n} \cdots \tau_{x_1}(y)$). Looking at 4.2, this means that we need to prove is that when $r(x,y) = (w,z)$ we have $\lambda_x \lambda_y = \lambda_w \lambda_z$ and $\tau_y \tau_x = \tau_z \tau_w$. (For details of this see \cite{ess}.)
\end{proof}

To give an explicit form for elements of $G_X$, we shall embed it into a larger group. Recall that if $N,H$ are groups and $\theta:H \to \text{Aut}(N)$ is a homomorphism (so $\theta$ defines an action of $H$ on $N$). Then the \emph{semidirect product} $N \rtimes_{\theta} H$ (often written $N \rtimes H$ if there is a natural choice of action $\theta$) is the set $H \times N$ with a group operation defined by:
\begin{equation}
(n_1, h_1)(n_2,h_2) = (n_1 \theta(h_1)(n_2),h_1 h_2).
\end{equation}

\begin{example}
The direct product $G \times H$ of groups $G$, $H$, is a special case of $G \rtimes_{\theta} H$ where $\theta$ is the trivial homomorphism.

The dihedral group $D_n$, which we know is generated by a single rotation $g$ and a single reflection $h$, is isomorphic to $\Z_n \rtimes_\theta \Z_2$, where $\theta(0)(k) = k$ and $\theta(1)(k) = -k$. The element $g^k h^l$ corresponds to the pair $(k,l)$.
\end{example}

\begin{definition}
Let $X$ be a set, $\Z^{(X)}$ be the free abelian group generated by $X$, and $Sym(X)$ be the group of bijections $X \to X$. Then let $M_X$ denote the semidirect product $\Z^{(X)} \rtimes Sym(X)$ where $Sym(X)$ acts on $\Z^{(X)}$ by:
\[ \theta \left( \sum n_x x \right) = \sum n_x \theta(x). \]
\end{definition}

\begin{proposition}[{\cite[Prop. 2.3]{ess}}]
Let $(X,r)$ be a non-degenerate involutive solution to the YBE, with $r(x,y) = (\lambda_x(y), \tau_y(x))$. Then the map $\phi:X \to M_X$ given by
\begin{equation}
\phi(x) = (x,\lambda_x),
\end{equation}
extends to a group homomorphism $\phi:G_X \to M_X$.
\end{proposition}

\begin{proof}
Since $r$ is involutive we have $\tau_y(x) = \lambda_{\lambda_x(y)}^{-1}(x)$. We need to show that whenever $r(x,y) = (w,z)$ we have:
\[ \phi(x)\phi(y) = (x,\lambda_x)(y,\lambda_y) = (w,\lambda_w)(z, \lambda_z) = \phi(w)\phi(z) . \]

By above, we have $w = \lambda_x(y)$ and $z = \lambda_{\lambda_x(y)}^{-1}(x) = \lambda_w^{-1}(x)$, and by proposition 4.2.2 we have $\lambda_x \lambda_y = \lambda_w \lambda_z$. Thus using 4.3 we calculate:
\begin{align*}
(x,\lambda_x)(y,\lambda_y) &= (x + \lambda_x(y), \lambda_x \lambda_y)\\
&= (x + w, \lambda_w \lambda_z)\\
&= (w + \lambda_w(\lambda_w^{-1}(x)), \lambda_w \lambda_z)\\
&= (w + \lambda_w(z), \lambda_w \lambda_z)\\
&= (w, \lambda_w)(z, \lambda_z).
\end{align*}
\end{proof}

Next consider the map $\pi = p_1 \circ \phi:G_X \to \Z^{(X)}$, where $p_1:\Z^{(X)} \rtimes Sym(X) \to \Z^{(X)}$ projects onto the first coordinate. This means that $\pi(g) = t$ where $\phi(g) = (t,\gamma)$ for some $\gamma \in Sym(X)$.

\begin{lemma}[{\cite[Prop. 2.5(b)]{ess}}]
The map $\pi$ is a bijection.
\end{lemma}

\begin{proof}
This is a very long proof, given in pages 10-11 of \cite{ess}.
\end{proof}

This lemma allows us to give a very useful formulation of $G_X$ in terms of $M_X$, since it says that for every $x \in G_X$, there is exactly one $t \in \Z^{(X)}$ such that $\phi(x) = (t, \gamma)$ for some $\gamma \in Sym(X)$, and that for every $t \in \Z^{(X)}$ there is some $(t, \gamma) \in \phi(G_X)$.

\begin{theorem}[{\cite[Prop. 2.4,2.5]{ess}}]
Let $(X,r)$ be a non-degenerate involutive solution to the YBE, with $r(x,y) = (\lambda_x(y), \tau_y(x))$.
\begin{enumerate}
\item[(i)] The structure group $G_X$ is isomorphic to the subgroup $\phi(G_X) \leq M_X$, which is generated by the set $S = \{(x, \lambda_x) \,:\, x \in X\}$.
\item[(ii)] The map $X \to Sym_X$ given by $x \mapsto \lambda_x$ for $x \in X$ extends to a map $\Z^{(X)} \to Sym_X$ so that we have:
\[ \phi(G_X) = \{(t, \lambda_t)\,: \, t \in \Z^{(X)}\} \subset M_X. \]
\end{enumerate}
\end{theorem}

\begin{proof}
(i) Since $\pi = p_1 \circ \phi$ is a bijection, we must have that $p_1$ is surjective and $\phi$ is injective. Thus $G_X \cong \phi(G_X)$. Since $X$ generates $G_X$, $\phi(X) = S$ generates $\phi(G_X)$.

(ii) Since $\pi$ is a bijection, for $t \in \Z^{(X)}$ we can define $\lambda_t$ by $\phi(\pi^{-1}(t)) = (t, \lambda_t)$, and every element of $\phi(G_X)$ can be written uniquely in this form.
\end{proof}

The results in this section so far were published long before the invention of braces and cycle sets. In more recent years, mathematicians have taken advantage of the representation of $G_X$ given to us by theorem 4.2.7 to show that $G_X$ can be given the structure of a brace. The result is implicit in \cite{ru2}, however we follow the proof given by Ced\'o, Jespers and Okni\'nski in \cite{cjo}.

\begin{theorem}[{\cite[Thm. 4.4]{cjo}}]
Let $(X,r)$ be a non-degenerate involutive solution to the YBE. The structure group $G_X$ is the adjoint group of a brace with additive group $\Z^{(X)}$.
\end{theorem}

\begin{proof}
By the above we can identify $G_X$ with $\{(t,\lambda_t)\,:\, t \in \Z^{(X)}\} \subset M_X$. Defining addition on $G_X$ by:
\[ (a, \lambda_a) + (b, \lambda_b) = (a+b, \lambda_{a+b}), \]
it is clear that the additive group is isomorphic to $\Z^{(X)}$. Since every element of $G_X$ is of the form $(t, \lambda_t)$ and $(a, \lambda_a)(b,\lambda_b) = (a + \lambda_a(b), \lambda_a \lambda_b)$, we have $\lambda_a \lambda_b = \lambda_{a + \lambda_a(b)}$. Thus we have $(a, \lambda_a) \circ (b,\lambda_b) = (a + \lambda_a(b), \lambda_{a + \lambda_a(b)})$, which we need to verify satisfies 2.6:
\begin{align*}
(a, \lambda_a) \circ ((b,\lambda_b) + (c,\lambda_c)) + (a,\lambda_a) &= (a, \lambda_a) \circ (b + c, \lambda_{b+c}) + (a,\lambda_a)\\
&= (a + \lambda_a(b + c), \lambda_{a + \lambda_a(b+c)}) + (a,\lambda_a)\\
&= (a + \lambda_a(b) + a + \lambda_a(c), \lambda_{a + \lambda_a(b) + a + \lambda_a(c)})\\
&= (a + \lambda_a(b), \lambda_{a + \lambda_a(b)}) + (a + \lambda_a(c), \lambda_{a + \lambda_a(c)})\\
&= (a, \lambda_a) \circ (b,\lambda_b) + (a, \lambda_a) \circ (c,\lambda_c).
\end{align*}
\end{proof}

Next we give the second part of \cite[Thm. 4.4]{cjo} with a novel, greatly simplified proof.

\begin{corollary}
The restriction to $X$ of the solution associated to the brace $G_X$ (as given in theorem 2.2.8) is exactly the solution $(X,r)$ of which $G_X$ is the structure group.
\end{corollary}

\begin{proof}
The original solution was $r(x,y) = (\lambda_x(y), \lambda_{\lambda_x(y)}^{-1}(x))$. The claim then, is that $\lambda_x(y) = x \circ y - x$ (applying the circle operation in $G_X$). Embedding $G_X$ into $M_X$ we have:
\[ (x, \lambda_x) \circ (y,\lambda_y) - (x,\lambda_x) = (x + \lambda_x(y), \lambda_{x + \lambda_x(y)} ) - (x, \lambda_x) = (\lambda_x(y), \lambda_{\lambda_x(y)}) \]
as required.
\end{proof}

In section 3.3 we saw that a non-degenerate cycle set $X$ admits a linear extension to $\Z^{(X)}$. Since linear cycle sets are equivalent to braces, this means we now have two constructions of brace structures on $\Z^{(X)}$. We now see that they are in fact the same, and as far as we are aware we are the first to explicitely relate these results (in \cite{ru4} Rump shows that the adjoint group of $\Z^{(X)}$ is isomorphic to $G_X$, but he uses a completely independent method from here).

\begin{proposition}
Let $X$ be a cycle set with associated QYBE solution $(X,R)$. The linear extension $\Z^{(X)}$ of a cycle set $X$ has adjoint group $G_{X}$, the structure group of the YBE solution $(X, r) = (X,Rp)$ (given by lemma 3.2.9).
\end{proposition}

\begin{proof}
The solution to the QYBE $(X, R)$ for the cycle set $X$ is given by:
\[ R(x,y) = (x^y, x^y \cdot y), \]
and so the corresponding solution to the YBE is
\[ r(x,y) = Rp(x,y) = R(y,x) = (y^x, y^x \cdot x). \]

Let $G_{X}$ denote the structure group of this solution. Since we have an isomorphism of additive groups $\pi:G_{X} \to \Z^{(X)}$ (lemma 4.2.6), it suffices to show that $(\Z^{(X)}, \circ)$ satisfies the relation 4.2 (since this means that $\pi$ is a homomorphism of adjoint groups, and therefore an isomorphism). The adjoint multiplication on $\Z^{(X)}$ is given by $a \circ b = a^b + b$. This gives $\Z^{(X)}$ a right brace structure while $G_{X}$ is a left brace, so we adapt 4.2 to:
\[ a \circ b = c \circ d \quad \text{ when } \quad r(b,a) = (d,c). \]

Thus we need to prove that $a \circ b = (a^b \cdot b) \circ (a^b)$. The left hand side is just $a^b + b$, and the right hand side is $(a^b \cdot b)^{a^b} + a^b = b + a^b$.
\end{proof}

\section{Constructing Cycle Sets}

After seeing many results from Rump on cycle sets and how they are useful for proving results in brace theory, we now explore an application of cycle sets outside of brace theory. We follow a method developed by Castelli, Pinto and Rump \cite{cpr} of constructing finite, indecomposable, involutive, non-degenerate set-theoretic solutions to the Yang-Baxter equation with a prime-power number of elements and cyclic permutation group. We decided to focus on this construction, as it highlights the fact that by using cycle sets we can construct solutions which we would not have been able to obtain from braces. For more constructions of finite braces see \cite{ru3}.

\subsection{Indecomposable Cycle Sets}

We first recall some of the basic properties of cycle sets and non-degenerate, involutive solutions. 

If $(X, \cdot)$ is a non degenerate cycle set, the map r from $X \times X \rightarrow X \times X$ given by $r(x,y)  :=  (\lambda_x(y), \tau_y(x))$, where $\lambda_x(y) \ := \sigma_x^{-1}(y)$ and $\tau_y(x) :=  \lambda_x(y) \cdot x$, is a non-degenerate involutive solution. Conversely, if $(X,r)$ is a non-degenerate involutive solution, the binary operation $\cdot$ given by $x \cdot y := \lambda_x^{-1}(y)$ for all $x,y \in X$ gives rise to a non-degenerate cycle set. 

\begin{definition}
A cycle set $(X, \cdot)$ is said to be \emph{retractable} if $|X| = 1$ or if there exists two distinct elements $x,y \in X$ such that $\sigma_x = \sigma_y$. (It is clear that a cycle set $X$ is retractable if and only if it is not irretractable in the sense of definition 3.3.11.)
\end{definition}

\begin{definition}
A non-degenerate cycle set $(X, \cdot)$ is called \emph{multipermutational of level m}, if m is the minimal non negative integer such that $\sigma^m(X)$ has cardinality one, where
\begin{equation}
\sigma^0(X) := X \ \ \text{and} \ \ \sigma^n(X) := \sigma(\sigma^{n-1}(X)), \ \ \text{for} \ \ n \geq 1
\end{equation}
\end{definition}

We recall that the permutation group $G(X)$ denotes the subgroup of $Sym(X)$ generated by the image of $\sigma:X \to Sym(X)$.

\begin{definition}
A subset $Y$ of $X$ is called \emph{$G(X)$-invariant} when $\sigma_x(Y) \subseteq Y$ for all $x \in X$
\end{definition}

The $G(X)$-invariant subsets of $X$ can also be thought of as unions of orbits of the action of $G(X)$ on $X$. Suppose $Y$ is $G(X)$-invariant and, $x \in Y$ and $y$ is in the orbit of $x$. Then there is some $z \in X$ such that $\sigma_z(x) = y$, and thus $y \in Y$. So $Y$ contains all the orbits of its elements. Conversely if $Y$ is the union of some orbits, then since the action of every element of $G(X)$ maps orbits to themselves, $Y$ is $G(X)$-invariant.

\begin{definition}
A non-degenerate cycle set $X$ is said to be \emph{decomposable} if there exists a partition $X = Y \sqcup Z$ such that $Y$ and $Z$ are non-empty $G(X)$-invariant subsets, and \emph{indecomposable} otherwise.
\end{definition}

\begin{proposition}
A cycle set $X$ is indecomposable if and only if $G(X)$ acts transitively on $X$.
\end{proposition}

\begin{proof}
Suppose $G(X)$ acts transitively on $X$, and suppose $X = Y \sqcup Z$ is a partition. Take $y \in Y$ and $z \in Z$: since $G(X)$ acts transitively there is some $\sigma_x$ such that $\sigma_x(y) = z$, and thus $Y$ is not $G(X)$-invariant. Hence there is no partition of $X$ into $G(X)$-invariant subsets, and so $X$ is indecomposable.

Conversely, suppose $G(X)$ does not act transitively on $X$, and take some $x \in X$. The orbit $O_x$ of $x$ is a $G(X)$-invariant subset, and since $X$ is partitioned into orbits, $X \setminus O_x$ is a union of orbits and therefore a $G(X)$-invariant subset. Since $G(X)$ does not act transitively, $O_x \subsetneq X$ and $X \setminus O_x \neq \emptyset$. Thus $X$ is decomposable as $O_x \sqcup X \setminus O_x$.
\end{proof}

For the remainder of the chapter we will be concerned with finite cycle sets, so non-degeneracy is automatic by theorem 3.3.14.

\begin{proposition}[{\cite[Prop. 1]{cpr}}]
Let $X$ be and indecomposable finite cycle set with $|X| > 1$, such that the permutation group $G(X)$ is abelian. Then $\sigma(X)$ is an indecomposable cycle set and $X$ is multipermutational.
\end{proposition}

\begin{proof}
As $X$ is finite we have that $\sigma(X)$ is a non-degenerate cycle set. Take $\sigma_x, \sigma_y \in \sigma(X)$. Since $X$ is indecomposible, $G(X)$ acts transitively and so there exists $z$ such that $\sigma_z(x) = z \cdot x = y$. Thus $\sigma_z \cdot \sigma_x = \sigma_y$, meaning that $G(\sigma(X))$ acts transitively on $\sigma(X)$.

As $G(X)$ is abelian we have $x \cdot (y \cdot z) = \sigma_x \sigma_y (z) = \sigma_y \sigma_x (z) = y \cdot (x \cdot z)$, for all $x,y,z \in X$.

If we show that $X$ is retractable, the above guarantees that $\sigma(X)$ is retractable since it fits the criteria placed on $X$, and so on. Thus $|X| > |\sigma(X)| > |\sigma^2(X)| \cdots$ is a strictly decreasing sequence of integers, meaning there is some $n$ such that $|\sigma^n(X)| = 1$, so $X$ is multipermutational.

Suppose for contradiction that $X$ is irretractible. Since $G(X)$ is abelian and acts transitively on $X$, it can be thought of as a transitive abelian subgroup of $S_{|X|}$ (the symmetric group on $|X|$ elements). herefore $|G(X)| = |X|$. Since $X$ is irretractible, $\sigma_x \neq \sigma_y$ for all $x,y \in X$, and thus there exists a unique $x \in X$ such that $\sigma_x = \id$. If $y \in X$ then
\[ y \cdot z = (x \cdot y) \cdot (x \cdot z) = (y \cdot x) \cdot (y \cdot z)\]
for all $z \in X$. Since $\sigma_y;z \mapsto y \cdot z$ is a bijection, every element of $X$ can be written in the form $y \cdot z$ for some $z$, and therefore $\sigma_{(y \cdot x)} = \id = \sigma_x$ for all $y \in X$. Since $X$ is irretractible, it follows that $y \cdot x = x$ for all $y \in X$ and hence $G(X)$ doesn't act transitively, a contradiction. 
\end{proof}

\subsection{Cycle Sets of Order $p^k$}

We are now going focus on indecomposable, finite cycle sets of order $p^k$, with cyclic permutation group. We will give a general construction of such cycle sets, but we need a few more lemmas first. For the following results, $X$ is such a cycle set and $G(X) = \langle \varphi \rangle$.

\begin{lemma}[{\cite[Lem. 3]{cpr}}]
We have that $n = |\sigma(X)|$ is the least integer such that $\sigma_{\varphi^n(x)} = \sigma_x$ for every $x \in X$. 
\end{lemma}

\begin{proof}
Let $x \in X$, and let $n \in \N$ be such that $\sigma_{\varphi^n(x)} = \sigma_x$. First, we show that $\sigma_{\varphi^{n+k}(x)} = \sigma_{\varphi^k}(x)$ for every $k \in \N$. 

Let $y_1, \cdots,y_j \in X$ such that $\varphi = \sigma_{y_1}, \cdots, \sigma_{y_j}$. Then we have that,
\begin{align*}
\sigma_{\varphi^{n+1}(X)} &= \sigma_{y_1 \cdot (y_2 \cdot ( \cdots(y_j \cdot \varphi^n(x))\cdots)))} \\
&= \sigma_{y_1} \cdot (\cdots (\sigma_{y_j} \cdot \sigma_{\varphi^n(x)}) \cdots ) \\
&= \sigma_{y_1} \cdot (\cdots (\sigma_{y_s} \cdot \sigma_x) \cdots) \\
&= \sigma_{\varphi(x)}
\end{align*}
so the result holds for $k = 1$. Now suppose it holds for $k \leq K$. We have:
\begin{align*}
\sigma_{\varphi^{n+(K+1)}(X)} &= \sigma_{y_1 \cdot (y_2 \cdot ( \cdots(y_j \cdot \varphi^{n+K}(x))\cdots)))} \\
&= \sigma_{y_1} \cdot (\cdots (\sigma_{y_j} \cdot \sigma_{\varphi^{n+K}(x)}) \cdots ) \\
&= \sigma_{y_1} \cdot (\cdots (\sigma_{y_s} \cdot \sigma_{\varphi^K(x)}) \cdots) \\
&= \sigma_{\varphi^{K+1}(x)}
\end{align*}
where in the third inequality we used the inductive hypothesis. Since $G(X)$ acts transitively and $\varphi$ generates $G(X)$, this implies that $n^* := \min\{n \in \N : \sigma_{\varphi^n(x)} = \sigma_x \}$ does not depend on the choice of element $x$, and that $\sigma(X) = \{\sigma_x,\sigma_{\varphi(x)},\cdots, \sigma_{\varphi^{n^*-1}(x)}\}$.
\end{proof}

\begin{corollary}
$|\sigma(X)|$ divides $|X|$.
\end{corollary}

\begin{proof}
If $G(X)$ is cyclic and acts transitively on $X$, then $|G(X)| = |X|$. If $n$ is such that $\sigma_{\varphi^n(x)} = \sigma_x$, then $\sigma_x = \sigma_{\varphi^n(x)} = \sigma_{\varphi^{2n}(x)} = \cdots = \sigma_{\varphi^{nr}(x)} = \cdots$ for $r \in \Z$.

Suppose $n$ is minimal for satisfying the condition, and does not divide the order of $\varphi$ (which we denote $o(\varphi)$), so $m := \gcd(n,o(\varphi)) < n$. There exists $r,s \in \Z$ such that $nr + o(\varphi)s = m$, and so $\varphi^m = \varphi^{nr + o(\varphi)s} = (\varphi^{o(\varphi)})^s \varphi^{nr} = \varphi^{nr}$. Thus if $n$ satisfies $\sigma_{\varphi^n(x)} = \sigma_x$, then $m < n$ satisfies $\sigma_{\varphi^m(x)} = \sigma_x$, contradicting minimality of $n$. Thus the minimal $n = |\sigma(X)|$ divides $o(\varphi) = |\langle \varphi \rangle| = |G(X)| = |X|$.
\end{proof}

\noindent
\emph{Remark:} It should be noted that $|\sigma(X)|$ divides $|X|$ for any finite, indecomposable cycle set $X$, by \cite[Lem. 1]{ccp}. However, the proof above is our own.

\begin{lemma}[{\cite[Lem. 4]{cpr}}]
There exists $x \in X$ such that $\langle \sigma_x \rangle = G(X)$.
\end{lemma}

\begin{proof}
Since $G(X)$ acts transitively on $X$ we have that $|G(X)| = |X| = p^k$. Moreover, for every $x \in X$ there exists a least $n_x \in \N$ such that $\sigma_x = \varphi^{n_x}$. 

If we suppose that $p\, |\, n_x$ for every $x \in X$ then $G(X)$ is contained in $\langle \varphi^p\rangle$ and hence $|G(X)|<p^k$. Thus there is some $x$ with $ p \not| \, n_x$. We have $\sigma_x^i = \varphi^{n_x i} = \id$ if and only if $p^k \, |\, n_x i$, so when $ p \not| \ n_x$, $\sigma_x^i = \id$ if and only if $p^k \, |\, i$. Thus $\langle \sigma_x \rangle = G(X)$.
\end{proof}

Since we are only considering indecomposable cycle sets of order $p^k$ with cyclic permutation groups, we can use lemma 5.2.2 to make our lives easier. From now on, for $x \in X$ such that $\langle \sigma_x \rangle = G(X)$, we set $\varphi := \sigma_x$ and $0 := x$. Then for $i \in \{1, \cdots , p^k - 1\}$, set $i := \varphi^i(0)$. Thus we have $X = \{0, \cdots, p^k - 1\}$, and $G(X) = \langle \varphi \rangle$ where $\varphi$ is just the cycle $(0 \, \cdots \, p^k - 1)$.

\begin{lemma}[{\cite[Lem. 5]{cpr}}]
Let $j_i \in \N$ be such that $\sigma_i = \varphi^{j_i}$ for each $i = 0,\cdots,p^k-1$. Then $j_i \equiv j_{i+|\sigma^s(X)|} \mod  |\sigma^{s-1}(X)|$ for every $i \in \{0,\cdots,p^k-1\}$ and $s\in\N$.
\end{lemma}

\begin{proof}
First we show the $s = 1$ case. By lemma 5.2.1, $\sigma_i = \sigma_{\varphi^{|\sigma(X)|}(i)} = \sigma_{i + |\sigma(X)|}$ and thus $j_i \cong j_{i + |\sigma(X)|} \mod |X|$, since $\sigma_i = \varphi^{j_i} = \varphi^{j_i + n|X|} = \varphi^{j_{i+\sigma(X)}} = \sigma_{i+\sigma(X)}$ for some $n$.

Since cyclic groups are abelian, proposition 5.1.6 tells us that $X$ is multipermutational, so we prove the $s>1$ case by induction on the multipermutation level of $X$. In the case where $X$ is multipermutation level 1 this is trivial, since $\sigma^{s}(X) = \sigma^{s-1}(X) = 1$.

Suppose the result holds for cycle sets of multipermutation level $\leq N$, and let $X$ be of multipermutation level $n = N+1$. Let $\varphi'$ be the cycle $(\sigma_0\, \cdots \, \sigma_{|\sigma(X)|-1})$, which by lemma 5.2.1 is well defined and has order $|\sigma(X)|$. It follows from the proof of proposition 5.1.6 that $|G(\sigma(X))| = |\sigma(X)|$, so $G(\sigma(X)) = \langle \varphi' \rangle$, and if $j_i'$ is such that $\sigma_{\sigma_{x_i}} = \phi'^{j'_i}$, then $j'_i \equiv j_i \mod |\sigma(X)|$.

Now $\sigma(X)$ is an indecomposable cycle set of prime power order (since $|\sigma(X)|$ divides $|X| = p^k$) and multipermutation level $n - 1 = N$, so by the inductive hypothesis we have that $j'_i \equiv j'_{i+|\sigma^{r+1}(X)|} \mod |\sigma^r(X)|$ for every $r \in \N$. Also, $j'_i \equiv j_i \mod |\sigma^r(X)|$ and $j'_{i+|\sigma^{r+1}(X)|} \equiv j_{i+|\sigma^{r+1}(X)|} \mod |\sigma^r(X)|$ since these hold mod $|\sigma(X)|$, and $|\sigma^r(X)|$ divides $|\sigma(X)|$ by corollary 5.2.2. Therefore it holds that $j_i \equiv j_{i+|\sigma^{r+1}(X)|} \mod |\sigma^r(X)|$ for every $i \in \{0,\cdots,p^k-1\}$ and $r \in \N$, which is exactly the $s > 1$ case where $s = r+1$.
\end{proof}

We are now ready to give the main result of the section: a method of taking a prime $p$, and natural numbers $n,k \in \N$, and constructing an indecomposable cycle set $X$ with $|X| = p^k$ and of multipermutation level $n$. Since when $X$ is of multipermutation level 1, $\sigma_x = \sigma_y$ for all $x,y$, this case is considered trivial and we will only consider $n > 1$.

\begin{theorem}[{\cite[Thm. 8]{cpr}}]
Let $X = \{0, \cdots, p^k - 1\}$ for some prime $p$ and $k \in \N$, $n \in \N \setminus \{1\}$ and $k = j_0 > \cdots > j_n = 0 \in \N \cup \{0\}$, and let $\{f_i\}_{i \in \{1, \cdots, n-1\}}$ be a family of functions:
\begin{equation}
f_i : \Z/{p^{j_i}}\Z \rightarrow \{0, \cdots, p^{j_{i-1} - j_i} - 1\}
\end{equation}
such that $f_i(0) = 0$ for every $i \in \{1, \cdots, n-1\}$ and the functions
\begin{align}
\psi_i &: \Z/{p^{j_i}}\Z \rightarrow \Z/{p^{j_{i-1}}}\Z \notag \\ 
&\ ;\ x \mapsto 1 + p^{j_{n-1}}f_{n-1}(x)+ \cdots + p^{j_i}f_i(x) 
\end{align}
are injective for each $i \in \{1,\cdots,n-1\}$ (when $\psi_i$ or $f_i$ are given an integer as an input, it is treated as its residue mod $p^{j_i}$). Finally, set $\varphi := (0\, \cdots \,p^k-1) \in Sym(X)$ and $\sigma_i := \varphi^{\psi_1(i)}$ for each $i \in X$. If 
\begin{equation}
x + 2\psi_1(y) \equiv y + 2\psi_1(x) \mod p^k,
\end{equation}
then $(X, \cdot)$ is an indecomposable cycle set of multipermutation level $n$, where $x \cdot y = \sigma_x(y)$. It has cyclic permutation group $G(X) = \langle \varphi \rangle$ and $|\sigma^i(X)| = p^{j_i}$ for each $i \in \{0,\cdots, n\}$.
\end{theorem}

\begin{proof}
We have $\sigma_x \in Sym(X)$, so the left multiplications are bijective. If $x \cdot y = \sigma_x(y)$, then $(x \cdot y) \cdot (x \cdot z) = \sigma_{x \cdot y} \sigma_x (z) = \sigma_{\sigma_x(y)} \sigma_x(z)$, and we have
\[ \sigma_{\sigma_x(y)}\sigma_{x} = \varphi^{\varphi^{\psi_1(x)}(y)} \varphi^{\psi_1(x)} = \varphi^{y + 2\psi_1(x)}, \]
with the second equality holding since $\varphi^{\psi_1(x)}(y) \equiv \psi_1(x) + y \mod p^k$. By assumption, $y + 2\varphi_1(x) \equiv x + 2 \varphi_1(y)$, and hence $X$ is a cycle set. By definition of the left multiplications it follows that $G(X) = \langle \varphi \rangle$, and thus $X$ is indecomposable since $G(X)$ acts transitively.

Now we show that $|\sigma^i(X)| = p^{j_i}$ for every $i \in \{0, \cdots, n\}$ by induction on $i$. For $i = 0$ we already know that $|\sigma^{0}(X)| = |X| = p^k = p^{j_0}$. For $i = 1$, we have $\sigma_x = \sigma_y$ whenever:
\begin{align*}
\varphi^{\psi_1(x)} = \varphi^{\psi_1(y)} & \iff \psi_1(x) \equiv \psi_1(y) \mod p^k\\
&\iff x \equiv y \mod p^{j_1},
\end{align*}
since $\psi_1$ is injective and defined on $\Z_{p^{j_1}}$. Hence $|\sigma(X)| = p^{j_1}$. 

Next, we suppose (as asserted in \cite{cpr}) that $\sigma^i(X)$ is isomorphic to the cycle set $\bar{X}_i = \{0,\cdots, p^{j_i} - 1\}$ with multiplication defined by:
\[ \sigma_x = \begin{cases} \bar{\varphi}_i^{\psi_{i+1}(x)} & i < n-1,\\ \bar{\varphi}_i & i=n-1. \end{cases} \]
for $x \in \bar{X}$, where $\bar{\varphi}_i = (0\, \cdots \, p^{j_i} - 1)$.

For $i < n-1$, $\bar{X}_i$ has been constructed as above, with $\bar{j}^i_1 = j_{i+1}$, so we can apply our proof of the $i=1$ case to see that $|\sigma^{i+1}(X)| = |\sigma(\bar{X}_i)| = p^{\bar{j}^i_1} = p^{j_{i+1}}$, giving the result for $2 \leq i \leq n-1$. For $i = n-1$, all the left multiplications of $\sigma^i(X)$ are equal to $\varphi$, and hence $|\sigma^{i+1}(X)| = |\sigma^{n}(X)| = 1 = p^{j_n}$, the result for $i = n$.

It remains to prove the assertion that $\sigma^i(X)$ is isomorphic to $\bar{X}_i$. (This is asserted without proof in \cite{cpr}, and so our proof is novel.) It suffices to show that $\sigma(X)$ is isomorphic to $\bar{X}_1$, since this can then be applied to $X= \bar{X}_1$ to give us $\sigma^2(X) = \sigma(\sigma(X)) \cong \sigma(\bar{X}_1) \cong \bar{X}_2$, and so on to obtain $\sigma^i(X) \cong \bar{X}_i$.

We saw above that $\sigma_x = \sigma_y$ whenever $x \equiv y \mod p^{j_1}$, so we map $\sigma_x \mapsto \bar{x}_1$, which denotes the unique element of $\bar{X}_1$ satisfying $\bar{x}_1 \equiv x \mod p^{j_1}$. This is clearly a bijection, so we need to check that $\sigma_x \cdot \sigma_y = \sigma_{x \cdot y} \mapsto \bar{x}_1 \cdot \bar{y}_1$, or equivalently that $x \cdot y \equiv \bar{x}_1 \cdot \bar{y}_1 \mod p^{j_1}$.

First assume $n > 2$ so that $\psi_2$ is well defined. From 5.3 we have
\begin{align*}
\psi_1(x) &= 1 + p^{j_{n-1}}f_{n-1}(x)+ \cdots + p^{j_2}f_2(x) + p^{j_1}f_1(x),\\
\psi_2(x) &= 1 + p^{j_{n-1}}f_{n-1}(x)+ \cdots + p^{j_2}f_2(x),
\end{align*}
meaning that $\psi_1(x) \equiv \psi_2(x) \mod p^{j_1}$, so in particular $\bar{\varphi}_1^{\psi_1(x)} = \bar{\varphi}_1^{\psi_2(x)}$. We also have $\bar{\varphi}_1(x) \equiv \varphi(x) \mod p^{j_1}$, since $p^{j_1}$ divides $p^k$. Thus (where the $\equiv$ is mod $p^{j_1}$):
\[ x \cdot y = \varphi^{\psi_1(x)}(y) \equiv \bar{\varphi}_1^{\psi_1(\bar{x}_1)}(\bar{y}_1)  = \bar{\varphi}_1^{\psi_2(\bar{x}_1)}(\bar{y}_1) = \bar{x}_1 \cdot \bar{y}_1, \]
as required. For the $n=2$ case, we have $\psi_1(x) = 1 + p^{j_1}f_1(x) \equiv 1 \mod p^{j_1}$, and so $\bar{\varphi}_1^{\psi_1(x)} = \bar{\varphi}_1$. Thus:
\[ x \cdot y = \varphi^{\psi_1(x)}(y) \equiv \bar{\varphi}_1^{\psi_1(\bar{x}_1)}(\bar{y}_1)  = \bar{\varphi}_1(\bar{y}_1) = \bar{x}_1 \cdot \bar{y}_1, \]
since $n-1 = 1$, so $\sigma_{x} = \bar{\varphi}_1$ for all $x \in \bar{X}_1$.
\end{proof}

We should note that in the proof given in \cite{cpr}, the authors define $K_{i,j}$ and $Q_{i,j}$ for $(i,j) \in X^2$, in addition to the $f_i$, $\psi_i$, $\varphi$ and $\sigma_x$. We simplified the conditions on the $Q_{i,j}$s, which are defined in terms of the $K_{i,j}$s, into the condition 5.4, making the proof much more efficient.

In \cite[Thm. 9, Cor. 10]{cpr}, they go on to prove that \emph{every} indecomposable cycle set of prime power order (and multipermutation level $\geq 2$ can be obtained in this fashion (up to isomorphism, of course). Here is an example of this result giving non-existence:

\begin{example}
Let $p = k = 2$, so $X = \{0,1,2,3\}$, and let  $n = 2$. Since $j_0 = 2$ and $j_2 = 0$ are fixed, we must have $j_1 = 1$. We only need to give $f_1:\Z/2\Z \rightarrow \{0,1\}$, and we need $\psi_1(x) = 1 + 2f(x)$ to be injective. Thus $f_1$ needs to be injective, so our only choices are $f_1(0) = 0, f_1(1) = 1$ and $f_1'(0) = 1, f_1'(1) = 0$.

For $f_1$, we have that $\psi_1(x) = 1 + 2f_1(x)$ is indeed injective:
\[ \psi_1(0) = 1, \quad \quad \psi_1(1) = 3. \]
However, we also need to check that $x + 2 \psi_1(y) \equiv y + 2 \psi_1(x) \mod 4$, and:
\begin{align*}
0 + 2 \psi_1(1) &= 2\cdot 3 = 6,\\
1 + 2 \psi_1(0) &= 1 + 2 \cdot 1 = 3,
\end{align*}
so this does not work.

For $f_1'$, we again have that $\psi_1(x) = 1 + 2f_1'(x)$ is injective:
\[ \psi_1(0) = 3, \quad \quad \psi_1(1) = 1. \]
Again, however, we also need to check that $x + 2 \psi_1(y) \equiv y + 2 \psi_1(x) \mod 4$, and:
\begin{align*}
0 + 2 \psi_1(1) &= 2\cdot 1 = 2,\\
1 + 2 \psi_1(0) &= 1 + 2 \cdot 3 = 7,
\end{align*}
so this does not work either. Thus there is no indecomposible cycle set of cardinality 4 and multipermutation level 2.
\end{example}

\addcontentsline{toc}{chapter}{Bibliography}
\bibliographystyle{abbrv}
\bibliography{project.bib}

\end{document}